\documentclass[final,12pt]{elsarticle}
\usepackage{geometry}
\usepackage{amssymb}
\usepackage{amsmath}
\usepackage{amsthm}

\usepackage{microtype}
\allowdisplaybreaks
\usepackage{graphicx}
\usepackage{url}
\usepackage{xcolor}
\usepackage{etoolbox}    
\usepackage{tablefootnote}
\usepackage{threeparttable}
\usepackage{tablefootnote}
\usepackage[hidelinks, colorlinks, allcolors=blue, bookmarksdepth=2]{hyperref}
\usepackage{framed,multirow}

\usepackage{amsfonts}
\usepackage{mathrsfs}
\usepackage{upgreek}
\usepackage{bbm}

\usepackage{epstopdf}
\usepackage{algorithmic}
\ifpdf
\DeclareGraphicsExtensions{.eps,.pdf,.png,.jpg}
\else
\DeclareGraphicsExtensions{.eps}
\fi
\usepackage[export]{adjustbox}

\usepackage{xspace}
\usepackage{bm}
\usepackage{physics}
\usepackage{stmaryrd}
\usepackage{enumitem}
\usepackage{subcaption}
\usepackage{mathtools}

\usepackage{booktabs}
\usepackage{multirow, multicol, makecell}
\usepackage{longtable}
\usepackage{threeparttable}

\usepackage{cleveref}

\newtheorem{proposition}{Proposition}
\newtheorem{remark}{Remark}

\crefname{assumption}{Assumption}{Assumptions}
\crefname{definition}{Definition}{Definitions}
\crefname{lemma}{Lemma}{Lemmas}
\crefname{remark}{Remark}{Remarks}
\crefname{theorem}{Theorem}{Theorems}
\crefname{proposition}{Proposition}{Propositions}
\crefname{section}{Section}{Sections}
\crefname{figure}{Fig.}{Figs.}
\crefname{equation}{}{}
\crefname{table}{Table}{Tables}
\crefname{appendix}{Appendix}{Appendices}
\def\stripappendixprefix#1~#2\relax{#2}
\creflabelformat{appendix}{#2\expandafter\stripappendixprefix#1\relax#3}

\renewcommand{\H}[0]{\mathsf{H}}
\newcommand{\D}[0]{\mathsf{D}}
\newcommand{\Q}[0]{\mathsf{Q}}
\newcommand{\E}[0]{\mathsf{E}}
\newcommand{\M}[0]{\mathsf{M}}

\newcommand{\F}[0]{\mathsf{F}}

\newcommand{\A}[0]{\mathsf{A}}
\newcommand{\I}[0]{\mathsf{I}}

\renewcommand{\S}[0]{\mathsf{S}}

\newcommand{\Z}[0]{\mathsf{Z}}
\renewcommand{\P}[0]{\mathsf{P}}
\newcommand{\B}[0]{\mathsf{B}}
\newcommand{\R}[0]{\mathsf{R}}
\newcommand{\N}[0]{\mathsf{N}}

\newcommand{\etal}[0]{{et~al.\@}\xspace}
\newcommand{\eg}[0]{{e.g.\@}\xspace}
\newcommand{\ie}[0]{{i.e.\@}\xspace}

\newcommand{\ignore}[1]{} 

\newcommand{\poly}[1]{\ensuremath{\mathbb{P}^{#1}}({\Omega}_k)}

\DeclareMathOperator*{\argmin}{arg\,min}
\newcommand{\der[2]}[]{\dfrac{{\rm d}#1}{{\rm d}#2}}%
\newcommand{\pder[2]}[]{\dfrac{\partial#1}{\partial#2}}
\newcommand{\fn}[1]{\mathcal{#1}}
\newcommand{\fnb}[1]{\bm{\mathcal{#1}}}

\journal{Journal of Computational Physics}

\begin{document}

\begin{frontmatter}



\title{Targeted Entropy-Stable Nonlinear Artificial Dissipation for High-Order Spectral-Element Summation-By-Parts Discretizations}


\author[uw]{Zelalem Arega Worku\corref{cor1}}
\ead{zelalem.worku@uwaterloo.ca}

\author[utias]{Alex Bercik}
\ead{alex.bercik@mail.utoronto.ca}

\author[uw]{David C. Del Rey Fern{\'a}ndez}
\ead{ddelreyfernandez@uwaterloo.ca}

\author[utias]{David W. Zingg}
\ead{david.zingg@utoronto.ca}

\cortext[cor1]{Corresponding author}

\affiliation[uw]{
  organization={Department of Applied Mathematics, University of Waterloo},
  addressline={200 University Avenue West},
  city={Waterloo},
  postcode={N2L 3G1},
  state={Ontario},
  country={Canada}
}

\affiliation[utias]{
  organization={Institute for Aerospace Studies, University of Toronto},
  addressline={4925 Dufferin St},
  city={Toronto},
  postcode={M3H 5T6},
  state={Ontario},
  country={Canada}
}

\begin{abstract}
Artificial dissipation improves the robustness of high-order methods for under-resolved flows, but can compromise accuracy and suppress physically relevant structures if its strength and spatial distribution are not appropriately controlled. We present a sensor-activated nonlinear artificial-dissipation framework, inspired by the residual correction method, for continuous and discontinuous summation-by-parts discretizations. In contrast to many existing spectral-element dissipation operators that either reduce the formal order of accuracy or target only the highest polynomial modes, the proposed framework permits design- and higher-order dissipation with explicit control over its magnitude, spatial distribution, activation, and convergence rate under mesh refinement. This modular construction also recovers several existing methods through appropriate choices of its components. The dissipation preserves discrete conservation and semidiscrete entropy stability whenever the base method satisfies these properties. It is dimensionally consistent, element local, and discretization agnostic, as the added term is constructed independently of the base spatial residual. For smooth solutions under suitable mesh assumptions, we establish parameter choices that retain the formal accuracy of the base discretization. Numerical experiments with summation-by-parts discretizations of the Burgers and Euler equations demonstrate substantially improved robustness for under-resolved flows and flows with shock waves while retaining high-order accuracy for smooth solutions, using the same parameter choices across several test problems. Overall, these results show that targeted dissipation can improve the robustness of high-order methods without excessive damping of resolved flow structures.
\end{abstract}

\begin{keyword}
Artificial dissipation\sep Entropy stability\sep High-order methods\sep Summation-by-parts\sep Nonlinear stabilization\sep Under-resolved flows
\end{keyword}

\end{frontmatter}



\section{Introduction} 
The potential of high-order methods for efficient simulation of computational fluid dynamics (CFD) problems is well established \cite{wang2013high}. Recent developments in entropy-conservative and entropy-stable discretizations have strengthened their nonlinear stability properties \cite{chen2017entropy,crean2018entropy}. Nevertheless, entropy stability alone does not ensure adequate damping of under-resolved modes. Supplementary artificial dissipation can reduce spurious oscillations, accelerate convergence to steady state, suppress unphysical growth of perturbations, and improve robustness by helping to preserve positivity of thermodynamic variables~\cite{pulliam2014fundamental,bercik2026practical,bercik2026stable}. Excessive dissipation, however, can degrade accuracy, suppress physically relevant structures, and reduce the maximum stable time step. An effective construction must therefore control where dissipation is applied, its magnitude, and its decay under mesh refinement, while retaining the desirable properties of the base discretization. 

Artificial dissipation can be introduced through interface penalties, separately added volume terms, or the discretization operators themselves. Interface dissipation commonly uses scalar penalties on conservative-variable jumps based on local wave speeds \cite{toro2013riemann,guermond2016fast,chen2017entropy}, or matrix penalties on entropy-variable jumps \cite{ismail2009affordable,derigs2017novel,winters2017uniquely,worku2023entropy}. For continuous Galerkin-type methods, solution-gradient jumps provide a related stabilization mechanism \cite{burman2004edge,crean2016investigation}, although they may adversely affect conditioning \cite{hicken2020entropy}. Dissipation can also be incorporated directly into derivative operators, as in upwind summation-by-parts (U-SBP) methods, which have been combined with flux-vector splitting for nonlinear conservation laws \cite{mattsson2017diagonal,ranocha2025robustness}. For demanding under-resolved flows, interface dissipation and U-SBP operators may still require supplementary volume dissipation. The literature on volume dissipation is extensive, particularly for shock capturing; Yu and Hesthaven \cite{yu2020study} compare several such models within the discontinuous Galerkin (DG) framework. Within the SBP finite difference framework, Mattsson \etal\ \cite{mattsson2004stable} developed artificial dissipation operators that preserve energy stability and design-order accuracy, and Stiernstr{\"o}m \etal\ \cite{STIERNSTROM} and Niemel{\"a} \cite{niemela2026residual} use conservation-law residuals to localize viscosity while retaining high-order accuracy in smooth regions. These constructions do not generally enforce a discrete entropy inequality for nonlinear systems, although Stiernstr{\"o}m \etal\ establish energy stability for scalar problems. 

Entropy-stable volume stabilizations closely related to the objective of this work include the local projection stabilization (LPS) of Hicken \cite{hicken2020entropy} and the high-order dissipation of Bercik, Craig Penner, and Zingg \cite{bercik2026stable}. Ranocha \etal\ \cite{ranocha2018stability} studied entropy-stable artificial dissipation and its relation to modal filtering within the SBP framework. In these approaches, the projection or derivative operator determines which solution components are damped, while polynomial annihilation properties determine the smooth-state accuracy of the stabilization. In the LPS formulation, to obtain a nontrivial element-local projection that annihilates the full degree-$p$ polynomial space, one requires more nodes than the number of polynomial basis functions. Therefore, on an element with the minimum number of nodes, one is forced to relax the annihilation condition. A related restriction occurs in the minimally nodal spectral-element construction of Bercik \etal\ \cite{bercik2026stable}, where the use of the highest-degree nontrivial derivative operator forces the dissipation operator to act only on the highest-degree polynomial mode, and to sacrifice one order of accuracy. This loss of polynomial exactness can be compensated by employing a mesh-dependent coefficient to provide the design-order scaling. This alternative motivates the current work in which an explicit nonlinear construction supplies the required smooth-state decay without relying on high-order polynomial annihilation of the distribution operator.

The residual correction framework of Abgrall \cite{abgrall2018general,abgrall2022reinterpretation} and its artificial-dissipation interpretation by Edoh \cite{edoh2024conservative} provide the algebraic structure used in this work. Related correction procedures have been used in different contexts to enforce a prescribed entropy balance or inequality for a given base spatial residual, \eg, \cite{chan2025artificial,picklo2025entropy,alberti2024entropy}. In contrast, we use the same algebraic structure to prescribe an additional, separately designed dissipative term for under-resolved solutions. The assigned budget can therefore remain nonzero when the base discretization is already entropy conservative or entropy stable. The resulting term has zero conservative moment and nonpositive entropy contribution independent of the base method. Its addition preserves conservation and semidiscrete entropy stability whenever the base method satisfies these properties, but does not generally make a non-entropy-stable base method entropy stable.

The contribution of this paper is a targeted nonlinear dissipation that combines a conservative-variable distribution, a separately specified nonnegative entropy-viscosity budget, and an operator-normalized smoothness sensor. The implemented distribution is based on first derivatives, while the nonlinear activation supplies adjustable high-order decay under mesh refinement. A Hill function of the smoothness sensor selects the local dissipation order between prescribed low- and high-order values. For smooth admissible states on shape-regular, quasi-uniform meshes with affine element mappings, we establish consistency estimates with fixed positive regularization and identify parameter choices for which the added residual is of design-order. The construction is element local and independent of the form of the base spatial residual, allowing the same formulation to be used with continuous and discontinuous SBP discretizations and compatible Galerkin formulations. Most of the required reference-element matrices can be precomputed and reused.

Numerical experiments with the Burgers and Euler equations assess the proposed dissipation for continuous and discontinuous SBP discretizations. The two- and three-dimensional isentropic vortex tests examine smooth-solution accuracy, conservation, and entropy dissipation, while the Kelvin--Helmholtz tests assess robustness for under-resolved vortical flows. The Burgers, Sod, Shu--Osher, and two-dimensional Riemann problems assess the performance of the proposed dissipation for solutions containing shocks. With a kinetic-energy-preserving base volume flux, the Taylor--Green vortex tests assess the influence of the added dissipation on kinetic-energy evolution. The same dissipation parameter-selection rules are used across the test problems. 

The paper is organized as follows. \Cref{sec:preliminaries,sec:base_discretization} present the preliminaries and base discretizations. \Cref{sec:artificial_dissipation,sec:dissipation_properties} develop and analyze the proposed dissipation. Numerical results and conclusions follow in \cref{sec:numerical_experiments,sec:conclusion}. \Cref{app:alternative_choices,app:implementation_summary} provide connections with existing volume-dissipation schemes and an implementation summary.

\section{Preliminaries}\label{sec:preliminaries}
\subsection{Hyperbolic PDE and entropy identities}
Let \(\Omega\subset\mathbb{R}^{d}\) be a compact and connected domain in \(d\) dimensions with piecewise smooth boundary \(\partial\Omega\). We consider the hyperbolic initial-boundary-value problem 
\begin{equation}\label{eq:hyperbolic_pde}
	\begin{aligned}
		\pder[\fnb{U}]{t}
		+
		\sum_{i=1}^{d}
		\pder[\fnb{F}^{(i)}(\fnb{U})]{x_i}
		 & =
		\bm{0},
		 &   & \bm{x}\in\Omega,\quad 0<t\le t_f,
		\\
		\fnb{U}(\bm{x},0)
		 & =
		\fnb{U}_0(\bm{x}),
		 &   & \bm{x}\in\Omega,
		\\
		\fn{B}(\fnb{U})
		 & =
		\fnb{U}_{\Gamma},
		 &   & \bm{x}\in\partial\Omega,\quad 0<t\le t_f,
	\end{aligned}
\end{equation}
where \(\fnb{U}\in\mathbb{R}^{n_c}\) is the conservative state variable, \(n_c\) is the number of conservation-law components, \(\fnb{F}^{(i)}\) is the flux in the \(x_i\)-direction, and \(\fn{B}\) is a boundary operator. Periodic boundary conditions are included in \cref{eq:hyperbolic_pde} by identifying periodic traces and treating the boundary facets as interior interfaces.

Let \(\fn{S}(\fnb{U})\) be a strictly convex mathematical entropy on the set of admissible states. The entropy variables and the symmetric positive definite entropy-variable Jacobian are, respectively,
\begin{equation}\label{eq:entropy_variables}
	\fnb{W}^{T}
	\coloneqq
	\pder[\fn{S}]{\fnb{U}},
	\qquad
	\A_0(\fnb{U})
	\coloneqq
	\pder[\fnb{U}]{\fnb{W}}.
\end{equation}
For each spatial direction, the entropy flux, \(\fn{G}^{(i)}\), satisfies 
\begin{equation}\label{eq:entropy_flux_compatibility}
	\fnb{W}^{T}
	\pder[\fnb{F}^{(i)}]{\fnb{U}}
	=
	\pder[\fn{G}^{(i)}]{\fnb{U}}.
\end{equation}
Further discussion on the existence of such entropy-entropy flux pairs satisfying \cref{eq:entropy_flux_compatibility} and their role in nonlinear stability can be found, for instance, in \cite{harten1983symmetric,tadmor1987numerical,tadmor2003entropy}. If \(\bm{n}\) is the outward unit normal on \(\partial\Omega\), we denote the normal entropy flux by
\begin{equation}\label{eq:normal_entropy_flux}
	\fn{G}^{n}(\fnb{U};\bm{n})
	\coloneqq
	\sum_{i=1}^{d} n_i\fn{G}^{(i)}(\fnb{U}).
\end{equation}
For smooth solutions, premultiplying \cref{eq:hyperbolic_pde} by \(\fnb{W}^{T}\), integrating over \(\Omega\), and integrating by parts give \begin{equation}\label{eq:continuous_entropy_balance}
	\der[]{t}
	\int_{\Omega}\fn{S}(\fnb{U})\dd{\Omega}
	=
	-\int_{\partial\Omega}\fn{G}^{n}\dd{\Gamma}.
\end{equation}
For entropy solutions with shocks, the equality in \cref{eq:continuous_entropy_balance} is replaced by the corresponding inequality \cite{dafermos2005hyperbolic}. For periodic boundary conditions, the right-hand side vanishes because the boundary contributions cancel pairwise. Finally, the entropy potential, which plays a central role in the design of entropy stable discretizations, is given by
\begin{equation}\label{eq:entropy_potential}
	\psi^{(i)}
	\coloneqq
	\fnb{W}^{T}\fnb{F}^{(i)}-\fn{G}^{(i)}.
\end{equation}

\subsection{Mesh and SBP operators}
Let \(\mathcal{T}_{h}\coloneqq \{\{\Omega_k\}_{k=1}^{n_e}:\Omega=\cup_{k=1}^{n_e}\Omega_k\}\) be a conforming, quasi-uniform, and shape-regular (\eg see \cite{worku2026convergence} Appendix F) tessellation of \(\Omega\) with \(n_e\) elements. The boundary of \(\Omega_k\) is \(\Gamma_k\), the set of interior facets is denoted by \(\Gamma^{I}\), and the set of physical boundary facets is denoted by \(\Gamma^{B}\). We also use \(\Gamma_k^{I}\coloneqq\Gamma_k\cap\Gamma^{I}\). We assume that every element \(\Omega_k\) contains \(n_p\) volume nodes. The integer \(n_f\) denotes the number of nodes on a facet. The element length scale is defined as \(h\coloneqq\max_{\Omega_k\in\mathcal{T}_h} \sup_{\bm{x},\bm{y}\in\Omega_k} \|\bm{x}-\bm{y}\|_2\).

We use \(\bm{1}_k\in\mathbb{R}^{n_p}\) for the vector of ones on \(\Omega_k\). For a C-SBP discretization with \(n_g\) global scalar degrees of freedom, \(\bm{1}_g\in\mathbb{R}^{n_g}\) denotes the corresponding global vector of ones. We use \(\I_m\in\mathbb{R}^{m\times m}\) for the identity matrix. If \(\mathsf{X}\) acts on scalar nodal values, then \(\overline{\mathsf{X}}\coloneqq\mathsf{X}\otimes\I_{n_c}\) denotes its componentwise extension to variables in systems of equations. We define the following:
\begin{equation}\label{eq:identity_matrices}
	\begin{aligned}
	\overline{\I}_k	&\coloneqq \bm{1}_k\otimes\I_{n_c},
	&\qquad
	\overline{\I}_g	&\coloneqq \bm{1}_g\otimes\I_{n_c},
	&\qquad
	\overline{\bm{1}}_k	&\coloneqq \bm{1}_k\otimes\bm{1}_{n_c}.
	\end{aligned}
\end{equation}

On element \(\Omega_k\), a degree-\(p\) diagonal-norm SBP operator in the \(x_i\)-direction satisfies
\begin{equation}\label{eq:local_sbp_operator}
	\D_{x_i k}=\H_k^{-1}\Q_{x_i k},
	\quad
	\Q_{x_i k}=	\S_{x_i k}+	\frac{1}{2}\E_{x_i k},
	\quad
	\S_{x_i k} =-\S_{x_i k}^{T},
	\quad
	\E_{x_i k}=\E_{x_i k}^{T}.
\end{equation}
Equivalently, we have \(\Q_{x_i k}+\Q_{x_i k}^{T}=\E_{x_i k}\), which is known as the SBP property. Here, \(\H_k\) is a symmetric positive definite norm (mass) matrix, and in this paper, we exclusively consider operators with diagonal norm matrices. We also exclusively use SBP diagonal-\(\E\) operators, which are SBP operators with collocated volume and facet nodes. For sufficiently smooth functions, the derivative operator, \(\D_{x_i k}\), satisfies
\[
    \left[\D_{x_i k}\bm v_k\right]_j
    =
    \pder[\fn{V}]{x_i}(\bm{x}_{j})
    +
    \fn{O}(h^p), \qquad \forall \,\fn{V}\in C^{p+1},
\]
where \(\bm v_k\) denotes the nodal restrictions of \(\fn{V}\) on \(\Omega_k\) \cite{crean2018entropy}. We remark that the notation \(\fn{O}(h^p)\) should be interpreted componentwise when it is applied to vector-valued quantities. We define the element volume and effective nodal spacing by
\begin{equation}\label{eq:element_volume}
	V_k
	\coloneqq
	\bm{1}_k^T\H_k\bm{1}_k,
	\qquad
	h_{{\rm eff},k}
	\coloneqq
	\left(\frac{V_k}{n_p}\right)^{1/d}.
\end{equation}
The boundary matrix, \(\E\), has the decomposition
\begin{equation}\label{eq:local_sbp_boundary}
	\E_{x_i k}=	\sum_{\gamma\subset\Gamma_k}\R_{\gamma k}^{T}\B_{\gamma}
	\N_{x_i\gamma k}\R_{\gamma k},
\end{equation}
where \(\R_{\gamma k}\) extrapolates quantities on volume nodes to facet nodes, \(\B_\gamma\) contains reference facet quadrature weights, and \(\N_{x_i\gamma k}\) contains the surface-scaled outward normal component in the \(x_i\)-direction. We also define the diagonal surface-normal magnitude by
\[
	[|\N_{\gamma k}|]_{mm}
	\coloneqq	
	\left(\sum_{i=1}^{d}[\N_{x_i\gamma k}]_{mm}^{2}	\right)^{1/2},
	\qquad
	|\overline{\N}_{\gamma k}|
	\coloneqq
	|\N_{\gamma k}|\otimes\I_{n_c}.
\]

For C-SBP discretizations, \(\P_k\) restricts global degrees of freedom to \(\Omega_k\), while \(\P_k^T\) applies the corresponding prolongation operation. The global C-SBP operators are assembled from element contributions through these restriction matrices. For example, for scalar nodal values,
\begin{equation}\label{eq:csbp_oper_def}
	\H_g	\coloneqq\sum_{\Omega_k\in\mathcal{T}_h}\P_k^T\H_k\P_k,
	\qquad
	\Q_{x_i}\coloneqq\sum_{\Omega_k\in\mathcal{T}_h}\P_k^T\Q_{x_i k}\P_k,
	\qquad
	\D_{x_i}\coloneqq\H_g^{-1}\Q_{x_i}.
\end{equation}
Thus, shared continuous nodes are represented by one global degree of freedom, while the local quadrature and derivative contributions are accumulated through the assembly. For systems, the same construction is applied componentwise using the barred matrices. For D-SBP discretizations, the element vectors remain independent and are coupled by SATs or numerical fluxes on \(\Gamma^{I}\cup\Gamma^{B}\). 

The SBP operators are constructed on a reference element, denoted by \(\widehat{\Omega}\), and mapped to each physical element \(\Omega_k\). We denote the coordinates on \(\Omega_k\) by \(\bm{x}=(x_1,\ldots,x_d)^T\) and those on \(\widehat{\Omega}\) by \(\bm{\xi}=(\xi_1,\ldots,\xi_d)^T\). Further details on multidimensional SBP operators and their construction on curved elements can be found in \cite{hicken2016multidimensional,fernandez2014review,svard2014review,chen2017entropy,crean2018entropy}. Furthermore, the standard physical SBP scaling on quasi-uniform, shape-regular meshes is discussed in Appendix~F of \cite{worku2026convergence}. In this paper, we use only the dimensional sizes \([\H_k]=L^d\), \([\overline{\D}_{x_i k}]=L^{-1}\), and \([\overline{\Q}_{x_i k}]=[\overline{\E}_{x_i k}]=L^{d-1}\), where \([\Z]\) denotes the physical dimension of \(\Z\) and \(L\) is the length scale.

\section{Base spatial discretization}\label{sec:base_discretization}
The artificial dissipation developed below is added to the semi-discrete spatial residual after the base spatial discretization has been formed. The analysis focuses on conservative entropy-stable base discretizations, and in particular on Hadamard-form two-point flux SBP discretizations. These are the default Euler base discretizations used in the numerical experiments. We also use the artificial dissipation with other types of discretizations, namely the entropy-stable split-form discretization for the Burgers equation and the direct divergence-form discretization of the Euler equations for the two-dimensional Riemann problem.

For completeness, and to introduce notation, we first present the Hadamard-form two-point flux SBP discretizations on a periodic domain. We use the sign convention \(\dd \bm{u}/\dd t=\bm r(\bm u)\), so \(\bm r\) approximates \(-\sum_{i=1}^{d}\partial\fnb{F}^{(i)}/\partial x_i\). For the D-SBP form, and for an interior facet \(\gamma=\Gamma_k\cap\Gamma_v\), we define
\begin{equation}\label{eq:Ekv_two_point}
	\E_{x_i}^{kv}\coloneqq\R_{\gamma k}^{T}\B_{\gamma}\N_{x_i\gamma k}\R_{\gamma v},
	\qquad
	\E_{x_i}^{vk}=-\left(\E_{x_i}^{kv}\right)^{T}.
\end{equation}
Furthermore, let \(\fnb{F}^{*}_{i}(\fnb{U}_L,\fnb{U}_R)\) be a symmetric and consistent two-point flux function. Entropy conservation is obtained when such a flux satisfies Tadmor's entropy consistency condition~\cite{tadmor1987numerical},
\begin{equation}\label{eq:tadmor_two_point}
	\left(
	\fnb{W}_L-\fnb{W}_R
	\right)^{T}
	\fnb{F}^{*}_{i}(\fnb{U}_L,\fnb{U}_R)
	=
	\psi_L^{(i)}
	-
	\psi_R^{(i)} .
\end{equation}
Following the Hadamard notation of Crean et al.~\cite{crean2018entropy} to apply the two-point flux in discretizations, the \((m,n)\)-th \(n_c\times n_c\) block of the two-point flux matrix is defined as
\begin{equation}\label{eq:Fstar_matrix}
	\left[
		\F^{*}_{i}\left(\bm u_k,\bm u_v	\right)	\right]_{mn}\coloneqq	2\operatorname{diag}
	\left(
	\fnb{F}^{*}_{i}	\left(\bm u_k^{(m)},\bm u_v^{(n)}\right)
	\right).
\end{equation}
With this notation, the D-SBP element-local two-point flux residual is written as
\begin{equation}\label{eq:local_two_point_residual}
	\bm r_k	= -\overline{\H}_{k}^{-1}\sum_{i=1}^{d}
	\left( \overline{\S}_{x_i k}\circ\F^{*}_{i}(\bm u_k,\bm u_k)\right) \overline{\bm{1}}_k
	-
	\frac{1}{2}	\overline{\H}_{k}^{-1} \sum_{\gamma\in\Gamma_k^{I}}
		\sum_{i=1}^{d}\left(\overline{\E}_{x_i}^{kv}\circ\F^{*}_{i}(\bm u_k,\bm u_v)
	\right)
	\overline{\bm{1}}_k,
\end{equation}
where \(\circ\) denotes the Hadamard (elementwise) multiplication. Further detail on the Hadamard-form construction and entropy estimate can be found in
\cite{hicken2016multidimensional,chen2017entropy,crean2018entropy,shadpey2020entropy}. For C-SBP discretizations, the solution is single-valued at shared nodes; hence, no interface coupling or SAT is introduced. Instead, the global SBP operator is assembled from local contributions, and the element-boundary terms in \cref{eq:local_two_point_residual} cancel pairwise at interior interfaces. Thus, with \(\bm u_k=\overline{\P}_k\bm u_g\), the element-local volume term contribution is simply the first term on the right-hand side (RHS) of \cref{eq:local_two_point_residual},
\begin{equation}\label{eq:csbp_local_two_point_residual}
	\bm q_k	=-\sum_{i=1}^{d}\left(\overline{\S}_{x_i k}	\circ\F^{*}_{i}(\bm u_k,\bm u_k)
	\right)
	\overline{\bm{1}}_k,
\end{equation}
and the global residual is assembled as
\begin{equation}\label{eq:csbp_two_point_residual}
	\der[\bm u_g]{t}=	\bm r_g, \qquad \bm r_g	=\overline{\H}_g^{-1}	\sum_{\Omega_k\in\mathcal{T}_h}	\overline{\P}_k^{T}	\bm q_k.
\end{equation}
Physical-boundary SATs can be added in the usual way; they contribute the corresponding boundary entropy estimate.

The conservation and entropy stability discussions below only require that the selected residual be conservative and entropy stable, \ie, it need not be based on two-point fluxes. For C-SBP discretizations, conservation and entropy stability imply
\begin{equation}\label{eq:selected_entropy_estimate_csbp}
	\overline{\I}_g^{T}\overline{\H}_g\bm r_g	=\bm{0},
	\qquad
	\bm{w}_g^{T}\overline{\H}_g\bm r_g \le0,
\end{equation}
while for D-SBP discretizations, we have
\begin{equation}\label{eq:selected_entropy_estimate_dsbp}
	\sum_{\Omega_k\in\mathcal{T}_h}
	\overline{\I}_k^{T}\overline{\H}_k\bm r_k=\bm{0},
	\qquad
	\sum_{\Omega_k\in\mathcal{T}_h}	\bm{w}_k^{T}\overline{\H}_k\bm r_k	\le0.
\end{equation}
For boundary conditions with nonzero entropy input, the zero upper bounds in \cref{eq:selected_entropy_estimate_csbp,eq:selected_entropy_estimate_dsbp} are replaced by the corresponding boundary entropy estimate. For D-SBP discretizations, interface dissipation can be added as a separate conservative SAT in the form of
\begin{equation}\label{eq:interface_entropy_dissipation}
	\bm r_k^{\operatorname{ES}}	=\bm r_k-\frac{1}{2}\overline{\H}_{k}^{-1}
	\sum_{\gamma\in\Gamma_k^{I}}\overline{\R}_{\gamma k}^{T}
	\overline{\B}_{\gamma}\,\bm{\ell}_{\gamma k}^{n},
\end{equation}
where \(\bm{\ell}_{\gamma k}^{n}\) is the normal-flux dissipation on facet \(\gamma\), written with the outward normal of \(\Omega_k\). In the case of the scalar local Lax--Friedrichs dissipation, \(\bm{\ell}_{\gamma k}^{n}\) takes the form
\begin{equation}\label{eq:LLF_dissipation}
	\bm{\ell}_{\gamma k}^{n}= \lambda_{\gamma}|\overline{\N}_{\gamma k}|
	\left(\overline{\R}_{\gamma k}\bm u_k
	-
	\overline{\R}_{\gamma v}\bm u_v
	\right),
\end{equation}
where \(\lambda_{\gamma}\) is computed with the unit normal using the Davis estimate \cite{davis1988simplified}, while \(|\overline{\N}_{\gamma k}|\) supplies the surface-normal scaling.

The preceding SAT dissipation is natural for D-SBP discretizations because the left and right interface states are independent. C-SBP discretizations do not have such double-valued interface states and therefore do not provide the same interface-SAT mechanism for damping under-resolved modes. Existing approaches for C-SBP discretizations therefore use either gradient-jump-based edge dissipations or specially constructed volume dissipation operators. Crean \etal~\cite{crean2016investigation} investigated edge stabilization for multidimensional simplex SBP discretizations of the Euler equations. The method was effective in their tests, and edge-stabilization ideas can be formulated to enforce entropy dissipation; however, the method was found to have a reduced explicit time step restriction. Hicken \cite{hicken2020entropy} later applied local-projection stabilization (LPS) to the entropy variables, which gives an entropy-dissipative and well-conditioned C-SBP discretization. The cost is that the element-local construction requires additional degrees of freedom beyond the minimal degree-\(p\) polynomial space; in one dimension, a nontrivial element-local operator that annihilates \(\poly{p}\) uses at least \(p+2\) nodes. Stable volume dissipation of the type developed by Bercik, Craig Penner, and Zingg \cite{bercik2026stable} provides another systematic C-SBP-compatible option and can be made energy- or entropy-stable. However, similar to the LPS approach, for minimal one-dimensional spectral-element operators, either one degree of polynomial exactness must be relaxed, or the accuracy must be recovered through coefficient scaling or a nonlinear design-order functional. These approaches are important baselines, but often face trade-offs for elements with the minimum number of nodes. In contrast, the dissipation presented below uses a local smoothness sensor to control activation, allowing it to target a wide range of polynomial modes without compromising accuracy or substantially increasing the number of time steps required to reach the final time in the numerical comparisons considered.

\section{Nonlinear high-order artificial dissipation}\label{sec:artificial_dissipation}
Hyperbolic conservation laws may admit multiple weak solutions after shocks form. The physically relevant solution is selected by an admissibility mechanism, classically through entropy inequalities and vanishing-viscosity limits \cite{lax1973hyperbolic,friedrichs1971systems}. This viewpoint also explains the role of artificial dissipation in shock-capturing methods; namely, nonnegative numerical viscosity regularizes under-resolved gradients and biases the discrete solution toward the entropy solution. In this work, artificial dissipation is used in the same admissibility sense, but it is added as an entropy-budgeted term in the semi-discrete spatial residual. When the base spatial discretization is entropy stable, the added dissipation preserves that entropy stability. The continuous model is used to identify the proper sign, scale, and physical dimensions of the discrete budget. A scalar conservative-variable viscosity is
\begin{equation}\label{eq:physical_viscosity}
	\pder[\fnb{U}]{t}+\sum_{i=1}^{d}\pder[\fnb{F}^{(i)}]{x_i}
	=
	\sum_{i=1}^{d}\pder[]{x_i}\left(\nu	\pder[\fnb{U}]{x_i}	\right),
	\qquad \nu\ge0,
\end{equation}
which is entropy dissipative for any convex entropy, up to the corresponding viscous entropy flux through the boundary. Equivalently, using \(\A_0\), defined in \cref{eq:entropy_variables}, this viscosity can be rewritten as \begin{equation}\label{eq:continuous_A0_viscosity}
	\pder[\fnb{U}]{t} +\sum_{i=1}^{d}\pder[\fnb{F}^{(i)}]{x_i}
	=\sum_{i=1}^{d}	\pder[]{x_i}\left(\nu\A_0\pder[\fnb{W}]{x_i}\right).
\end{equation}
On a periodic domain, or for boundary conditions that make the entropy flux vanish, multiplying by \(\fnb{W}^T\), integrating over \(\Omega\), and integrating by parts gives
\begin{equation}\label{eq:continuous_A0_entropy_balance}
	\der[]{t}\int_{\Omega}\fn{S}(\fnb{U})\dd{\Omega}
	=
	-\sum_{i=1}^{d}\int_{\Omega}\nu	\left(\pder[\fnb{W}]{x_i}\right)^T\A_0
	\pder[\fnb{W}]{x_i}	\dd{\Omega}	\le0.
\end{equation}
We determine the scaling of the constant viscosity coefficient \(\nu\) in \cref{eq:physical_viscosity} by comparing the dimensions of the RHS and left-hand side (LHS) of the equation. Since the RHS contains two spatial derivatives, it has dimension \([\nu][U]/L^2\), which must match the first-order time derivative of the LHS with dimension \([U]/T\), where \(T\) denotes time. Thus, a dimensionally consistent viscosity has dimension \([\nu]=L^2/T\), which implies that
\begin{equation}\label{eq:hyperbolic_viscosity_scale}
	\nu\sim\lambda h,
\end{equation}
where \(\lambda\) is a characteristic speed and \(h\) is a local length scale. We note that, here and below, dimensional relations are interpreted componentwise for vector quantities, \eg, \([U]\) and \([W]\) denote the physical dimensions of corresponding components of \(\fnb{U}\) and \(\fnb{W}\), respectively.

\subsection{A general algebraic structure for entropy-stable artificial dissipation}\label{sec:entropy_stable_artificial_dissipation_form}
The algebraic structure of the proposed nonlinear artificial dissipation is motivated by the residual correction method of Abgrall \etal~\cite{abgrall2018general,abgrall2022reinterpretation}. We add an artificial dissipation term to the base spatial residual that is designed to have zero conservative moment and a prescribed entropy-production moment by construction.

Let \(\bm r_g\) be the base spatial residual with norm matrix \(\overline{\H}_g\), conservative state \(\bm u_g\), and entropy variables \(\bm w_g\). For C-SBP discretizations, \(\bm u_g\), \(\bm w_g\), and \(\bm r_g\) are assembled global vectors, and \(\overline{\P}_k\) restricts them to \(\Omega_k\). For D-SBP discretizations, these vectors are formed by stacking their element-local counterparts, and \(\overline{\P}_k\) extracts the \(k\)-th element block. In either case, we set \(\bm u_k=\overline{\P}_k\bm u_g\) and \(\bm w_k=\overline{\P}_k\bm w_g\). We add the element-local artificial dissipation, \(\bm d_k\), to the spatial residual as
\begin{equation}\label{eq:local_artificial_dissipation_update}
	\der[\bm u_g]{t}=\bm r_g+\overline{\H}_g^{-1}
	\sum_{\Omega_k\in\mathcal{T}_h}
	\overline{\P}_k^T\bm d_k,
	\qquad
	\bm d_k	\coloneqq - \bm\alpha_{k,\epsilon}\varepsilon_k,
\end{equation}
where \(\bm\alpha_{k,\epsilon} (\bm u_k)\) is the local distribution vector, controlling how the artificial dissipation is applied in each element, and \(\varepsilon_k(\bm u_k)\) is the scalar amount of entropy dissipation assigned to \(\Omega_k\), controlling the total amount of dissipation applied in each element. The subscript \(\epsilon\) indicates that the entropy normalization is regularized by \(\epsilon_{\rm mach}V_k^2\) to avoid division by zero, where \(\epsilon_{\rm mach}>0\) denotes machine precision. 

For the artificial dissipation to preserve the semidiscrete conservation and entropy-stability conditions \eqref{eq:selected_entropy_estimate_csbp} and \eqref{eq:selected_entropy_estimate_dsbp}, we construct a nonnegative entropy-dissipation budget, \(\varepsilon_k\geq0\), and a distribution vector, \(\bm\alpha_{k,\epsilon}\), satisfying
\[
	\overline{\I}_k^T\bm\alpha_{k,\epsilon}=\bm0,
	\qquad
	0 \leq \bm w_k^T\bm\alpha_{k,\epsilon} \leq 1.
\]
We guarantee the satisfaction of this second condition (a bounded entropy-moment) by setting \(\bm\alpha_{k,\epsilon}\) to be of the following form,
\begin{align}
	\bm\alpha_{k,\epsilon}
	&\coloneqq
	\frac{b_k\bm a_k}{b_k^2+\epsilon_{\rm mach}V_k^2}
	=
	\frac{(b_k/V_k)(\bm a_k/V_k)}{(b_k/V_k)^2+\epsilon_{\rm mach}}, \quad	
    \epsilon_{\rm mach}>0, \label{eq:local_alpha_definition}
	\\
	 b_k&\coloneqq\bm w_k^T\bm a_k, \label{eq:b_k}
\end{align}
where \(\bm a_k\) is a yet-to-be-determined vector that controls the spatial distribution of the dissipation within the element.
The equivalent form in \cref{eq:local_alpha_definition} shows that the regularization is applied to the volume-normalized entropy moment \(b_k/V_k\). 
The normalization by $V_k$ comes from $b_k$ approximating an integral quantity, which will become clear following the choices made in \cref{sec:distribution-vector}. In particular, in \cref{sec:smooth_scaling} we will show that under mesh refinement, \(\bm a_k=\fn{O}(h^{d-1})\). Then if \(\bar{\bm w}_k\) is a representative constant entropy-variable state on \(\Omega_k\), since \(\overline{\I}_k^T\bm a_k=\bm0\) and \(\bm w_k-\overline{\I}_k\bar{\bm w}_k=\fn{O}(h)\), we obtain
\[
	b_k
	=
	(\bm w_k-\overline{\I}_k\bar{\bm w}_k)^T\bm a_k
	=
	\fn{O}(h)\fn{O}(h^{d-1})
	=
	\fn{O}(h^d).
\]
Hence, on a shape-regular, quasi-uniform mesh family, \(b_k/V_k=\fn{O}(1)\). This prevents the regularizer $\epsilon_{\rm mach}$ from becoming dominant solely because the element volume decreases under refinement. Conversely, if \(b_k/V_k\) is genuinely small (\eg near a locally constant state where $\bm w_k \approx \overline{\I}_k\bar{\bm w}_k$), the distribution is smoothly attenuated.

The scalar element-local entropy budget is defined as
\begin{equation}\label{eq:epsilon_k}
	\varepsilon_k
	\coloneqq
	\kappa_{\rm diss}\theta_k^{\beta_k}b_{\nu,k},
	\qquad
	\kappa_{\rm diss}>0,
\end{equation}
 where \(\theta_k\in[0,1]\) is a smoothness sensor that estimates the under-resolved content in the element, \(b_{\nu,k}\geq0\) is an entropy-viscosity factor, and \(\beta_k\) is the sensor-dependent exponent, while \(\kappa_{\rm diss}>0\) is an optional parameter that controls the overall dissipation magnitude; we use \(\kappa_{\rm diss}=1\). The entropy budget should approximate the entropy-viscosity of \eqref{eq:continuous_A0_entropy_balance} and therefore should implicitly contain an approximation to $\nu$ following \eqref{eq:hyperbolic_viscosity_scale}.
 
 The aim is to damp solution variations that are inadequately resolved by the spatial discretization, while minimizing dissipation of resolved structures. The following subsections specify four independent components:
\begin{enumerate}
    \item the distribution vector \(\bm\alpha_{k,\epsilon}\), which through $\bm{a}_k$, controls where dissipation is applied within each element (\cref{sec:distribution-vector});
    \item the sensor \(\theta_k\), which uses local reference-derivative content as a proxy for under-resolution to control activation (\cref{sec:volume_smoothness_sensor});
    \item the exponent \(\beta_k\), which adaptively modifies the selectivity of \(\theta_k\) and its convergence rate under mesh refinement (\cref{sec:activation_regime_indicator});
    \item the entropy-viscosity factor \(b_{\nu,k}\), which sets the budget scale, with \(\kappa_{\rm diss}\) controlling its overall magnitude (\cref{sec:entropy_viscosity_budget}).
\end{enumerate}

\subsection{Distribution vector} \label{sec:distribution-vector}

In this work, we set \(\bm a_k\) to mimic the spatial distribution of typical finite-difference volume dissipations (e.g.\,\cite{bercik2026stable}). For \(s\geq1\), we define the following reference matrices recursively by
\begin{equation}\label{eq:reference_smoothness_matrix}
	\hat{\M}_0
	\coloneqq
	\overline{\H}_{\xi},
	\qquad
	\hat{\M}_{s}
	\coloneqq
	\sum_{i=1}^{d}
	\overline{\D}_{\xi_i}^{T}
	\hat{\M}_{s-1}
	\overline{\D}_{\xi_i}.
\end{equation}
In addition to derivatives along each coordinate direction, here we have deliberately set \(\hat{\M}_s\) to contain all mixed derivatives of total order \(s\). It is symmetric positive semidefinite and satisfies \(\hat{\M}_s\overline{\I}_k=\bm0\) for \(s\geq1\). We use this matrix family with derivative orders \(s_a=1\) for the distribution vector~\(\bm a_{k}\), and \(s_\theta\) for the smoothness sensor introduced later. We then define the unscaled distribution vector as
\begin{equation}\label{eq:scaled_reference_distribution_a}
	\bm a_k
	\coloneqq
	\M_k\bm u_k,
	\qquad
	\M_k
	\coloneqq
	\mu_k\hat{\M}_{s_a}
	=
	\mu_k\hat{\M}_1.
\end{equation}
The factor \(\mu_k\) gives \(\M_k\) the same physical scaling as \(\sum_{i=1}^d\overline{\D}_{x_i k}^T\overline{\H}_k\overline{\D}_{x_i k}\) and is defined by
\begin{equation}\label{eq:distribution_scaling_factor}
	\mu_k
	\coloneqq
	\left(\frac{V_k}{\widehat V}\right)^{(d-2)/d},
	\qquad
	\widehat V
	\coloneqq
	\bm{1}_k^T\H_\xi\bm{1}_k.
\end{equation}
The normalization by \(\widehat V\) removes dependence on the size of the chosen reference element. We deliberately construct \(\bm a_k\) from conservative variables and evaluate its entropy moment, \(b_k=\bm w_k^T\bm a_k\), using entropy variables. In~\cite{bercik2026stable}, derivatives are applied to the entropy variables because \eqref{eq:scaled_reference_distribution_a} then naturally forms a positive definite quadratic quantity upon contraction with $\bm w_k^T$ from the left when computing the entropy moment. However, entropy variables are not always the best variables from which to build the distribution vector. For the Euler equations, entropy variables contain factors involving the inverse of pressure. Thus, as the pressure becomes small, applying \(\M_k\) directly to \(\bm w_k\) can produce very large distribution coefficients. Even when the scalar entropy moment is fixed, this can increase pointwise update magnitudes and make the assembled operator stiff. A related warning appears in \cite{derigs2017novel} indicating that dissipation based on entropy-variable jumps can break down in demanding pressure-jump tests when the entropy Jacobian average is chosen naively, producing negative densities and pressures. Although the quantities \(b_k\) and \(\varepsilon_k\) use entropy variables, they are scalar entropy moments, not raw entropy-variable distribution vectors. Thus, they do not affect a single node as a distribution vector would. Despite the use of mixed variables in the construction of the dissipation coefficients, we note that the complete update is dimensionally consistent, as discussed further in \cref{sec:dimensional_consistency}. 

\begin{remark}\label{rem:graph_distribution_safeguard}
The derivative-based distribution can have a vanishing entropy moment for nonconstant admissible states. In this case, the regularized formulation suppresses the added dissipation while retaining conservation and semidiscrete entropy stability. A graph-Laplacian construction introduced in the companion residual-correction study \cite{workurcm} provides an alternative distribution vector whose entropy moment is positive for nonconstant admissible states and vanishes only for constant states, provided that the entropy is differentiable and strictly convex on a convex admissible set. An optional safeguard is to replace the derivative-based distribution vector by this graph-Laplacian distribution vector whenever \(b_k^2<C_b\epsilon_{\rm mach}V_k^2\), with \(C_b>1\), for example \(C_b=100\), and recompute \(b_k\) before evaluating the distribution coefficient in \cref{eq:local_alpha_definition}. In exploratory comparisons (not shown here), using the graph-Laplacian distribution improved robustness for the Kelvin--Helmholtz problem, but produced stronger attenuation of the post-shock density waves in the Shu--Osher problem and greater kinetic-energy decay in the Taylor--Green vortex problem than the derivative-based construction. These observations indicate a problem-dependent tradeoff between robustness and resolution. We therefore retain the derivative-based distribution in this study. The conditional graph safeguard is not used in the numerical experiments reported here.
\end{remark}

\subsection{Solution smoothness sensor}\label{sec:volume_smoothness_sensor}
The next ingredient we address is a bounded nondimensional sensor that preferentially activates the entropy dissipation budget in under-resolved regions. The sensor should be element local, should vanish on constant states, and decay at high order under mesh refinement for smooth solutions. Once again, there are several forms that such a $\theta_k$ could take to satisfy these three conditions, and there is a broad literature regarding the development of such sensors, \eg~\cite{persson2006subcell, klockner2011viscous, guermond2011entropy, hennemann2021provably, ducros_sensor,ivan2014high}. Here, we develop a novel sensor, loosely inspired once more by typical volume dissipation, that we have found to perform well in numerical experiments.

For systems, the conservative variables have different physical dimensions and can have very different magnitudes. A quadratic form applied directly to \(\bm u_k\) would therefore depend on the choice of units and could be dominated by one component for dimensional rather than smoothness reasons. We avoid this by scaling the conservative variables before applying the reference smoothness matrix. Let
\begin{equation}\label{eq:state_scaling_matrix}
	\mathsf{C}_{k}^{u} \coloneqq\I_{n_p}\otimes	\operatorname{diag}
	\left(c_{1,k},\ldots,c_{n_c,k}\right),
	\qquad c_{\ell,k}>0,
\end{equation}
be an element-constant scaling matrix with the physical dimensions of the conservative components. Define the dimensionless scaled state
\begin{equation}\label{eq:scaled_state}
	\widetilde{\bm u}_k	\coloneqq
	\left(\mathsf{C}_{k}^{u}\right)^{-1} \bm u_k.
\end{equation}
The matrix \(\mathsf{C}_{k}^{u}\) is held constant on \(\Omega_k\). Thus, the sensor measures the same elementwise variation that \(\hat{\M}_{s_\theta}\) would ``see'' in the conservative variables, but after nondimensionalizing the components. In particular, the scaling changes the relative weights of the components but does not change the polynomial degree detected by \(\hat{\M}_{s_\theta}\). For the Euler equations with conservative variables
\[
	\fnb{U}	=\left(\rho,\rho v_1,\ldots,\rho v_d,e\right)^T,
\]
a direct implementation can use the element-average state to set
\begin{equation}\label{eq:euler_state_scales}
	c_{\rho,k}=\bar\rho_k,
	\qquad
	c_{\rho v_i,k}=c_{\rho,k}z_k,
	\qquad
	c_{e,k}=c_{\rho,k}z_k^2,
	\qquad
	z_k=|\bar{\bm v}_k|+\bar a_k,
\end{equation}
where \(\bar a_k=\sqrt{\gamma\bar p_k/\bar\rho_k}\) is the sound speed of the element-average state and \(|\bar{\bm v}_k|=(\sum_{i=1}^d\bar v_{i,k}^2)^{1/2}\). Since the element-average state is a positive-weight average of admissible Euler states, it is admissible, and hence \(\bar\rho_k>0\) and \(\bar p_k>0\). Therefore, \(\bar a_k>0\), \(z_k>0\), and \(c_{\ell,k}>0\).

We design the sensor to compare the amount of high-order reference-derivative content of the scaled state with the magnitude of the scaled state. For this purpose, we use the reference matrix \(\hat{\M}_{s_\theta}\) from \cref{eq:reference_smoothness_matrix}. The two dimensionless quantities used for this comparison are
\begin{align}
	S_k & \coloneqq \widetilde{\bm u}_k^T\hat{\M}_{s_\theta}\widetilde{\bm u}_k,
	\label{eq:volume_sensor_numerator}
	\\
	N_k & \coloneqq	\widetilde{\bm u}_k^T\overline{\H}_\xi\widetilde{\bm u}_k.
	\label{eq:volume_sensor_reference_norm}
\end{align}
A ``raw'' sensor can then be set as the bounded ratio
\begin{equation}\label{eq:volume_sensor}
	\vartheta_k	\coloneqq\frac{S_k}{S_k+N_k+\epsilon_{\rm mach}},
	\qquad
	\epsilon_{\rm mach}>0.
\end{equation}
Since \(S_k\) and \(N_k\) are constructed from the dimensionless scaled state using reference-element operators, they contain no physical element-volume factor; therefore, \(\epsilon_{\rm mach}\) enters without additional scaling by \(V_k\) and prevents division by zero. Since \(S_k\ge0\) and \(N_k\ge0\), the raw sensor satisfies
\begin{equation}\label{eq:volume_sensor_raw_bounds}
	0\le\vartheta_k<1 .
\end{equation}
It switches off on constants because \(S_k=0\). In smooth nonconstant regions, a Taylor expansion under an affine mapping, together with polynomial exactness, shows that applying a sequence of \(s_\theta\) reference derivative operators to \(\widetilde{\bm u}_k\) gives an \(\fn{O}(h^{s_\theta})\) term. Therefore, the product of two such terms in \eqref{eq:volume_sensor_numerator} gives
\begin{equation}\label{eq:volume_sensor_smooth_scaling}
	S_k	=\fn{O}\!\left(h^{2s_\theta}\right),
	\qquad
	N_k=\fn{O}(1),
\end{equation}
where the estimate for \(N_k\) follows from boundedness of the scaled state, and hence \(\vartheta_k=\fn{O}\big(h^{2s_\theta}\big)\) in smooth regions.
In contrast, under-resolved volume data can make \(\vartheta_k=\fn{O}(1)\).

\subsubsection{Operator-normalization}\label{sec:operator_normalized_volume_sensor}

The bounded raw sensor \(\vartheta_k\) is useful, but its value depends on the spectral radius of \(\hat{\M}_{s_\theta}\), which can vary with \(p\), \(s_\theta\), and nodal distribution, and make the dissipation difficult to calibrate. To mitigate this, we first define
\begin{equation}\label{eq:Lambda_s}
	\Lambda_{s_\theta} \coloneqq \lambda_{\max} \left(\hat{\M}_{s_\theta}, \overline{\H}_\xi\right)
	=
	\max_{\bm z\ne\bm0} \frac{\bm z^T\hat{\M}_{s_\theta}\bm z} {\bm z^T\overline{\H}_\xi\bm z},
\end{equation}
which is a dimensionless and positive quantity. For each reference operator, polynomial degree \(p\), and sensor order \(s_\theta\), this small symmetric generalized eigenproblem is solved once and the resulting value is reused. The operator-normalized sensor is then defined as the Rayleigh quotient \(S_k/N_k\) divided by its upper bound, \(\Lambda_{s_\theta}\). We obtain
\begin{equation}\label{eq:volume_sensor_operator_normalized}
	\bar{\vartheta}_k \coloneqq	\frac{S_k}{\Lambda_{s_\theta} N_k+\epsilon_{\rm mach}}.
\end{equation}
Since \(S_k\le\Lambda_{s_\theta} N_k\), the normalized sensor satisfies \(0\le\bar{\vartheta}_k<1\) and removes the operator-dependent scaling of \(S_k\) induced by \(\hat{\M}_{s_\theta}\).
\subsubsection{Removing the constant state}

As we have written in \eqref{eq:volume_sensor_operator_normalized}, the sensor measures the relative strength of high-order solution derivatives. The total energy, as measured by \(N_k\), includes both the fluctuation and mean flow. This can make the sensor insensitive to small under-resolved perturbations riding on a large constant state. We therefore further modify \eqref{eq:volume_sensor_operator_normalized} to use a normalization state with the mean flow removed, as measured by the same discrete inner product used in the sensor. Related normalizations based on deviations from an average are used in the existing indicators of~\cite{klockner2011viscous,guermond2011entropy,ivan2014high}.

In the reference norm, the componentwise constants are spanned by \(\overline{\I}_k\), so we use the \(\overline{\H}_\xi\)-orthogonal projection of \(\widetilde{\bm u}_k\) onto this constant space to compute the mean as
\[
	\bar{\widetilde{\bm u}}_k
	= \argmin_{\bm v\in\mathbb{R}^{n_c}}
	\|\widetilde{\bm u}_k-\overline{\I}_k\bm v\|_{\H_\xi}^2.
\]
Taking the Fr{\'e}chet derivative with respect to \(\bm v\), and applying the first-order optimality condition gives \(\overline{\I}_k^{T}\overline{\H}_\xi\left(\widetilde{\bm u}_k-\overline{\I}_k\bm v\right)=\bm0\). Hence, we compute the reference mean and fluctuation vectors as
\begin{align}
	\bar{\widetilde{\bm u}}_k & \coloneqq
	\left(\overline{\I}_k^{T}\overline{\H}_\xi\overline{\I}_k\right)^{-1}
	\overline{\I}_k^{T}\overline{\H}_\xi\widetilde{\bm u}_k,
	\label{eq:scaled_state_reference_mean}
	\\
	\widetilde{\bm u}'_k & \coloneqq \widetilde{\bm u}_k
	-\overline{\I}_k\bar{\widetilde{\bm u}}_k.
	\label{eq:scaled_state_fluctuation}
\end{align}
At the minimizer \(\bm v=\bar{\widetilde{\bm u}}_k\), the optimality condition gives \(\overline{\I}_k^{T}\overline{\H}_\xi\widetilde{\bm u}'_k=\bm0\). Therefore, the fluctuation $\widetilde{\bm u}'_k$ has zero discrete mean, meaning the element-average component of the scaled state has been removed. Note that the numerator $S_k$ in \eqref{eq:volume_sensor_operator_normalized} can equivalently be computed using \(\widetilde{\bm u}'_k\) since \(\hat{\M}_{s_\theta}\overline{\I}_k=\bm0\). Let \(L_{\rm ref}>0\) be a fixed reference length used to nondimensionalize the spatial coordinates. For \(p\geq2\), we define the sensor used in the dissipation budget, \cref{eq:epsilon_k}, and its squared fluctuation norm as
\begin{align}
	\theta_k
	&\coloneqq
	\frac{(\widetilde{\bm u}'_k)^T \hat{\M}_{s_\theta} \widetilde{\bm u}'_k}
	{\Lambda_{s_\theta}\left[N'_k+\widehat V(h_{{\rm eff},k}/L_{\rm ref})^2\right]}
	=
	\frac{S_k}
	{\Lambda_{s_\theta}\left[N'_k+\widehat V(h_{{\rm eff},k}/L_{\rm ref})^2\right]},
	\label{eq:volume_sensor_operator_normalized_fluctuation}
	\\
	N_k'
	&\coloneqq
	(\widetilde{\bm u}'_k)^T\overline{\H}_\xi\widetilde{\bm u}'_k.
	\label{eq:volume_sensor_fluctuation_norm}
\end{align}
The added term in the denominator is the squared reference-quadrature norm of a constant field with magnitude \(h_{{\rm eff},k}/L_{\rm ref}\). It sets a threshold at which oscillations can safely be ignored, roughly the scale at which the fluctuations match the effective nodal spacing. This keeps the denominator positive, replacing the need for $\epsilon_{\rm mach}$.
For the numerical experiments, the spatial coordinates are nondimensional and we take \(L_{\rm ref}=1\). The generalized-eigenvalue bound gives \(S_k\leq\Lambda_{s_\theta}N'_k\); hence, the sensor satisfies
\begin{equation}\label{eq:volume_sensor_bounds}
	0\le\theta_k<1 .
\end{equation}
\begin{remark}\label{rem:relative_roughness_sensor}
    The sensor \(\theta_k\) measures local oscillatory variation in the scaled conservative state relative to the local grid resolution. It is not a physics-based classifier and does not separately identify shocks, shear layers, under-resolved physical structures, or oscillations of purely numerical origin. Instead, it responds to all such features through the local variation measured by \(\hat{\M}_{s_\theta}\).
\end{remark}

For fixed \(p\), \(L_{\rm ref}\), and reference operators, shape regularity gives \(h_{{\rm eff},k}=\fn{O}(h)\). Recalling that \(S_k=\fn{O}(h^{2s_\theta})\), we can write
\begin{equation}\label{eq:regularized_sensor_smooth_bound}
	0
	\le
	\theta_k
	\le
	\frac{S_k}{\Lambda_{s_\theta}\widehat V(h_{{\rm eff},k}/L_{\rm ref})^2}
	=
	\fn{O}(h^{2s_\theta-2}).
\end{equation}
If \(N'_k=0\), the generalized-eigenvalue bound gives \(S_k=0\), while the mesh-dependent term keeps the denominator positive, and therefore \(\theta_k=0\). For \(p\geq2\), we use
\begin{equation}\label{eq:s_theta}
	s_\theta=\left\lceil\frac{p}{2}\right\rceil+1.
\end{equation}
Therefore, the construction considered here applies for \(p\geq2\), and \cref{eq:regularized_sensor_smooth_bound} gives \(\theta_k=\fn{O}(h^{2\lceil p/2\rceil})\) in smooth solution regions. For \(p=1\), we use the normalized operator given in \cref{eq:volume_sensor_operator_normalized}.

\begin{remark}\label{rem:entropy_based_volume_sensors}
	The same sensor construction can also be applied to other thermodynamic quantities such as density, pressure, or entropy in the context of the Euler equations. For the entropy, for instance, we can define
	\[
		[\bm \sigma_k]_{m} \coloneqq	\log\!\left(\frac{p_{k,m}}{\rho_{k,m}^{\gamma}}\right),
			\qquad 	m\in \{1,\dots,n_p\},
	\]
	and use either \(\bm \sigma_k\) or its mean-subtracted fluctuation in place of the scaled conservative state in the sensor. The Rayleigh-quotient bound still gives a sensor between zero and one. This alternative is not used in the numerical experiments.
\end{remark}

\subsubsection{Adaptive dissipation order}\label{sec:activation_regime_indicator}

In the entropy budget \eqref{eq:epsilon_k}, the sensor $\theta_k$ is raised to the power $\beta_k$ to exaggerate the activation, such that regions of $\theta_k \sim 1$ remain strongly dissipated, while regions with $\theta_k \ll 1$ receive vanishingly small dissipation under mesh refinement. It may be inappropriate, however, to expect the dissipation applied to very small activations (\eg in regions of smooth flow) to decay at the same rate as small but important activations in under-resolved regions of the flow (\eg dissipating small-scale structures below mesh resolution). Therefore, we introduce a mechanism to separately scale the sensor based on these two separate regimes.
To distinguish weak from strong sensor responses without evaluating a separate indicator, we map \(\theta_k\) through the Hill function
\begin{equation}\label{eq:theta_exponent_indicator}
	\chi_k
	\coloneqq
	\frac{\theta_k^{m_\theta}}
	{\theta_k^{m_\theta}+c_\theta^{m_\theta}},
	\qquad
	 0<c_\theta<1,
	\quad
	m_\theta>0.
\end{equation}
The indicator satisfies \(0\leq\chi_k\leq1\) and \(\chi_k=1/2\) when \(\theta_k=c_\theta\); hence, \(c_\theta\) is the transition value and \(m_\theta\) controls its sharpness. For a smooth state, \cref{eq:regularized_sensor_smooth_bound} gives
\begin{equation}\label{eq:theta_indicator_smooth_bound}
	0\leq\chi_k
	\leq
	\left(\frac{\theta_k}{c_\theta}\right)^{m_\theta}
	=
	\fn{O}\!\left(h^{(2s_\theta-2)m_\theta}\right).
\end{equation}

Let \(p_{\rm L}\) and \(p_{\rm H}\) be prescribed low and high dissipation-decay orders, respectively. We define the element-local order and sensor exponent by
\begin{align}
	p_k
	&\coloneqq
	(1-\chi_k)p_{\rm H}+\chi_kp_{\rm L},
	\qquad 0<p_{\rm L}\leq p_{\rm H},
	\label{eq:local_dissipation_order}\\
	\beta_k
	&\coloneqq\frac{p_k}{2s_\theta-2}.
	\label{eq:sensor_exponent}
\end{align}
Thus, \(\theta_k\ll c_\theta\) gives \(p_k\approx p_{\rm H}\), whereas \(\theta_k\gg c_\theta\) gives \(p_k\approx p_{\rm L}\). By \cref{eq:regularized_sensor_smooth_bound}, this gives an adaptive sensor response of $\theta_k^{\beta_k} = \order{h^{p_k}}$.

This Hill function indicator separates three distinct regions: shocks and very large oscillations, where $\theta_k^{\beta_k} = \order{1}$; small but non-negligible oscillations, where $\theta_k^{\beta_k} = \order{h^{p_{\rm L}}}$; and very smooth oscillations requiring little dissipation, where $\theta_k^{\beta_k} = \order{h^{p_{\rm H}}}$. To ensure design order, in this work we set $p_{\rm L} = p$, but take inspiration from the dissipation arising from upwind SATs, which scales as $\order{h^{2p+1}}$, to set $p_{\rm H} = 2p$.

\subsection{Entropy-viscosity budget}\label{sec:entropy_viscosity_budget}

One can define a local entropy-viscosity budget that directly approximates the continuous counterpart in \cref{eq:continuous_A0_entropy_balance} through
\begin{equation}\label{eq:bnu}
	\widetilde{b}_{\nu,k} \coloneqq \sum_{i=1}^{d}
	\left(\overline{\D}_{x_i k}\bm w_k\right)^T
	\overline{\H}_k	\operatorname{diag}{\left(\bm{\nu}_{i,k} \right)} \A_{0,k} \left(	\overline{\D}_{x_i k}\bm w_k\right),
\end{equation}
where $\bm{\nu}_{i,k}$ is an approximation to the scalar viscosity coefficient $\nu$ in direction $x_i$ at each node, and \(\A_{0,k}\) is a symmetric positive definite block diagonal matrix containing an approximation of \(\partial\fnb{U}/\partial\fnb{W}\) at each node. We denote this quantity as the ``physical'' reference budget, against which a cheaper budget below is compared. Note this approximates a scalar viscosity through the use of a scalar $\nu$; we defer investigations of more sophisticated matrix viscosities (\eg as in \cite{bercik2026stable}) to future work.

\subsubsection{Cheaper entropy-viscosity budget}\label{sec:cheaper_entropy_viscosity_budget}
The entropy budget in \cref{eq:bnu} requires physical derivative matrices, a nodal approximation to the viscosity coefficient $\nu$, and a nodal block diagonal matrix \(\A_{0,k}\). However, we can argue that the magnitude of this budget is arbitrary, as we do not know the exact value of \(\nu > 0\) in the continuous formulation, \cref{eq:physical_viscosity}. This motivates a simplified and cheaper construction of \(\widetilde{b}_{\nu,k}\) that retains a similar structure. 

Our first approximation is to use a constant coefficient $\nu_k$ in place of a nodal vector. Recalling the scaling \eqref{eq:hyperbolic_viscosity_scale}, we define
\[
\nu_k \coloneqq \lambda_k h_{{\rm eff},k},
\]
where \(h_{{\rm eff},k}\) is defined in \cref{eq:element_volume} and \(\lambda_k\) is an element-wise upper bound on the characteristic speed. The product \(\lambda_k h_{{\rm eff},k}\) is an elementwise approximation for a nodal viscosity scale; it preserves the required units and mesh-size scaling but not directional geometric information. The second approximation is to similarly use a constant entropy-Jacobian matrix based on a single cell-average state,
\[
	\bar{\A}_{0,k} \coloneqq \A_0(\bar{\bm u}_k), \qquad \bar{\bm u}_k \coloneqq	\frac{1}{V_k}\overline{\I}_k^T \overline{\H}_k \bm u_k .
\]
Finally, we compute derivatives directly on the reference element instead of using physical operators, yielding the following cheap approximation to \cref{eq:bnu}:
\begin{equation}\label{eq:bnu_cheap}
	b_{\nu,k}	\coloneqq \lambda_k h_{{\rm eff},k} \mu_k \sum_{i=1}^{d} \left(\overline{\D}_{\xi_i}\bm w_k \right)^T \left( \H_\xi\otimes\bar{\A}_{0,k}\right)
	\left( \overline{\D}_{\xi_i}\bm w_k	\right),
\end{equation}
where the scaling factor, \(\mu_k\), is defined in \cref{eq:distribution_scaling_factor}.

This approximation is nonnegative because \(\bar{\A}_{0,k}\) is symmetric positive definite, and it vanishes for constant entropy variables. It also has the same physical dimensions as the entropy production budget in \cref{eq:epsilon_k} and \cref{eq:bnu}, \ie,
\begin{equation} \label{eq:epsilon_k_units}
    [b_{\nu,k}] =	\frac{[W]\,[U]\, L^{d}}{T}. 
\end{equation}
We use \cref{eq:bnu_cheap} for all numerical results in this paper, and therefore reserve the notation \(b_{\nu,k}\) for this implemented budget. In implementation, \cref{eq:bnu_cheap} only requires applying the fixed reference derivative matrices to \(\bm w_k\), applying the small \(n_c\times n_c\) matrix \(\bar{\A}_{0,k}\) at each volume node, and reducing over nodes and coordinate directions.

\begin{remark} \label{rem:ref2phys}
    Here we have chosen to explicitly track the reference-to-physical scaling through the use of constants $h_{{\rm eff},k}$ and $\mu_k$ in each element. For highly skewed elements, it may instead be desirable to directly use the metric terms at each node, \(\fn{J}_k\,\partial\xi_i/\partial x_j\), where \(\fn{J}_k\coloneqq\det\!\left(\partial x_i/\partial\xi_j\right)\), providing a natural local scaling. For example, using
    \[
    \nu_{i,k} = \left\vert \sum_j \lambda_{x_j} \left( \fn{J}_k \pder[\xi_i]{x_j} \right) \right\vert
    \]
    at any node to approximate the viscosity coefficient along direction $\xi_i$ with characteristic speeds in directions $x_j$ yields the same units and $h$-scaling as \eqref{eq:bnu_cheap} without requiring explicit terms for $h_{{\rm eff},k}$ and $\mu_k$.
\end{remark}

\section{Properties of the dissipation operator}\label{sec:dissipation_properties}

We next establish the dimensional consistency, high-order accuracy in smooth regions, conservation, and entropy stability of the proposed dissipation operator.

\subsection{Dimensional consistency}\label{sec:dimensional_consistency}
The state scaling in \cref{eq:scaled_state} and the length ratio \(h_{{\rm eff},k}/L_{\rm ref}\) make \(\theta_k\), \(\chi_k\), \(p_k\), and \(\beta_k\) dimensionless. Taking $\varepsilon_k$ as defined in \eqref{eq:epsilon_k}, and using the result of \eqref{eq:epsilon_k_units} together with the fact that $\kappa_{\rm diss}$ and $\theta_k$ are dimensionless, we obtain
\begin{equation}\label{eq:epsilon_dimension}
	[\varepsilon_k]	= [b_{\nu,k}] = [W][U]L^d/T.
\end{equation}
Taking $\bm\alpha_{k,\epsilon}$ as defined in \eqref{eq:local_alpha_definition}, regardless of the dimensions of $\bm a_k$ in \eqref{eq:scaled_reference_distribution_a}, we obtain
\[
	[\bm\alpha_{k,\epsilon}]=1/[W].
\]
Therefore, overall for the dissipation as defined in \eqref{eq:local_artificial_dissipation_update} we obtain
\[
	\left[\overline{\H}_g^{-1} \overline{\P}_k^T\bm\alpha_{k,\epsilon}\varepsilon_k	\right]
	=
	L^{-d} \frac{1}{[W]} \frac{[W] [U] L^d}{T} = [U]/T,
\]
which shows that the artificial dissipation has the same dimension as \(\partial{\fnb{U}}/\partial{t}\).

\subsection{High-order accuracy in smooth regions}\label{sec:smooth_scaling}

We examine the asymptotic scaling of the artificial-dissipation term when applied to smooth solutions. The following result distinguishes the low order guaranteed for any indicator value from the high order recovered as the indicator vanishes.

\begin{proposition}\label{prop:ad_design_order_consistency}
Let \(p\geq2\), and assume that the mesh family is shape regular and quasi-uniform, the reference-to-physical element mappings are affine, and the solution remains smooth and admissible. 
Let the distribution vector be as defined in \eqref{eq:local_alpha_definition} and \eqref{eq:scaled_reference_distribution_a}, the entropy budget be as defined in \eqref{eq:epsilon_k} with sensor \eqref{eq:volume_sensor_operator_normalized_fluctuation}, the exponent be \eqref{eq:sensor_exponent} with some indicator satisfying \(0\leq\chi_k\leq1\), and the viscosity budget be as defined in \eqref{eq:bnu_cheap}.
Then the element-local and assembled artificial-dissipation contributions in \cref{eq:local_artificial_dissipation_update} satisfy
\[
	\|\bm d_k\|_\infty
	=
	\fn{O}(h^{d+p_{\rm L}}),
	\qquad
	\left\|
	\overline{\H}_g^{-1}
	\sum_{\Omega_k\in\mathcal T_h}
	\overline{\P}_k^T\bm d_k
	\right\|_\infty
	=
	\fn{O}(h^{p_{\rm L}}).
\]
For the indicator in \cref{eq:theta_exponent_indicator}, \(p_{\rm L}\) may be replaced by \(p_{\rm H}\) in both estimates.
\end{proposition}

\begin{proof}
We fix \(p\), \(p_{\rm L}\), \(p_{\rm H}\), \(L_{\rm ref}>0\), \(c_\theta\), \(m_\theta\), the reference operators, and \(\epsilon_{\rm mach}>0\) independently of \(h\), and take \(\kappa_{\rm diss}=\fn{O}(1)\). Shape regularity and quasi-uniformity give constants \(c,C>0\), independent of \(k\) and \(h\), such that
\[
	ch^d\le V_k\le Ch^d,
	\qquad
	\mu_k=\fn{O}(h^{d-2}),
	\qquad
	h_{{\rm eff},k}=\fn{O}(h).
\]
As stated in \cref{eq:volume_sensor_smooth_scaling}, the sensor-numerator estimate is given by
\[
	S_k=\fn{O}(h^{2s_\theta}).
\]
Furthermore, since \(\hat{\M}_1\) annihilates constants, smooth sampled data satisfy
\[
	\hat{\M}_1\bm u_k
	=
	\hat{\M}_1
	\left(\bm u_k-\overline{\I}_k\bm c_k\right)
	=
	\fn{O}(h),
\]
where \(\bm c_k\) is any representative constant state on \(\Omega_k\). Therefore, using \cref{eq:scaled_reference_distribution_a,eq:distribution_scaling_factor}, we have
\begin{equation} \label{eq:reference_distribution_a_scaling}
	\|\bm a_k\|_\infty
	=
	\fn{O}(h^{d-2})\fn{O}(h)
	=
	\fn{O}(h^{d-1}).
\end{equation}
Similarly, smoothness of the entropy-variable map gives \(\overline{\D}_{\xi_i}\bm w_k=\fn{O}(h)\). Hence, \cref{eq:bnu_cheap} gives
\[
	b_{\nu,k}
	=
	\fn{O}(h)\,
	\fn{O}(h^{d-2})\,
	\fn{O}(h^2)
	=
	\fn{O}(h^{d+1}).
\]
For a constant solution, \(S_k=0\), \(\bm a_k=\bm0\), and \(b_{\nu,k}=0\), so the complete artificial-dissipation update vanishes.

Furthermore, the arithmetic--geometric mean inequality gives
\[
	b_k^2+\epsilon_{\rm mach} V_k^2
	\ge
	2\sqrt{\epsilon_{\rm mach}}V_k|b_k|,
\]
and therefore
\[
	\frac{|b_k|}
	{b_k^2+\epsilon_{\rm mach} V_k^2}
	\le
	\frac{1}
	{2\sqrt{\epsilon_{\rm mach}}V_k}.
\]
Consequently, \eqref{eq:local_alpha_definition} followed by \eqref{eq:reference_distribution_a_scaling} and $V_k = \order{h^d}$, give
\begin{equation}\label{eq:regularized_distribution_smooth_bound}
	\|\bm\alpha_{k,\epsilon}\|_\infty
	\le
	\frac{\|\bm a_k\|_\infty}
	{2\sqrt{\epsilon_{\rm mach}}V_k}
	=
	\fn{O}(h^{-1}).
\end{equation}
This estimate also includes \(b_k=0\), for which \(\bm\alpha_{k,\epsilon}=\bm0\).

By \cref{eq:regularized_sensor_smooth_bound}, we have \(\theta_k=\fn{O}(h^{2s_\theta-2})\). Also, \cref{eq:local_dissipation_order,eq:sensor_exponent} give \(\beta_k\geq p_{\rm L}/(2s_\theta-2)>0\). Therefore, we obtain
\[
	0
	\le
	\theta_k^{\beta_k}
	\le
	\theta_k^{p_{\rm L}/(2s_\theta-2)}
	=
	\fn{O}(h^{p_{\rm L}}).
\]
It follows that the assigned entropy budget satisfies
\begin{equation}\label{eq:epsilon_convergence}
	\varepsilon_k
	=
	\kappa_{\rm diss}\theta_k^{\beta_k}b_{\nu,k}
	=
	\fn{O}(h^{d+1+p_{\rm L}}).
\end{equation}
Combining this estimate with \cref{eq:regularized_distribution_smooth_bound} gives the element-local scaling
\[
	\|\bm d_k\|_\infty
	=
	\|\bm\alpha_{k,\epsilon}\varepsilon_k\|_\infty
	=
	\fn{O}(h^{-1})\fn{O}(h^{d+1+p_{\rm L}})
	=
	\fn{O}(h^{d+p_{\rm L}}).
\]
The inverse of the diagonal norm matrix scales as \(\fn{O}(h^{-d})\), and the number of element contributions at each node is bounded independently of \(h\). Hence, the assembled artificial-dissipation residual vector satisfies
\[
	\left\|
	\overline{\H}_g^{-1}
	\sum_{\Omega_k\in\mathcal T_h}
	\overline{\P}_k^T
	\bm\alpha_{k,\epsilon}\varepsilon_k
	\right\|_\infty
	=
	\fn{O}(h^{-d})
	\fn{O}(h^{-1})
	\fn{O}(h^{d+1+p_{\rm L}})
	=
	\fn{O}(h^{p_{\rm L}}).
\]
To obtain the sharper estimate, let \(\Delta p\coloneqq p_{\rm H}-p_{\rm L}\). From \cref{eq:local_dissipation_order,eq:sensor_exponent}, we have
\[
	(2s_\theta-2)\beta_k
	=
	p_{\rm H}-\Delta p\,\chi_k.
\]
By \cref{eq:regularized_sensor_smooth_bound,eq:theta_indicator_smooth_bound}, there exist constants \(C_\theta,C_\chi>0\), independent of \(h\), such that, for sufficiently small \(h\),
\[
	\theta_k
	\leq
	C_\theta h^{2s_\theta-2},
	\qquad
	\chi_k
	\leq
	C_\chi h^{(2s_\theta-2)m_\theta},
\]
where we may take \(C_\chi=c_\theta^{-m_\theta}C_\theta^{m_\theta}\). The case \(\theta_k=0\) is immediate, so suppose \(0<\theta_k\leq1\). For sufficiently small \(h\), we have \(C_\theta h^{2s_\theta-2}<1\) and \(p_{\rm H}-\Delta p\,C_\chi h^{(2s_\theta-2)m_\theta}>0\). The preceding bound on \(\chi_k\) therefore gives
\[
	\beta_k
	\geq
	\frac{p_{\rm H}-\Delta p\,C_\chi h^{(2s_\theta-2)m_\theta}}
	{2s_\theta-2}.
\]
Since \(0<\theta_k\leq C_\theta h^{2s_\theta-2}<1\), we obtain
\[
	\begin{aligned}
	\theta_k^{\beta_k}
	&\leq
	\left(C_\theta h^{2s_\theta-2}\right)^{
	\frac{p_{\rm H}-\Delta p\,C_\chi h^{(2s_\theta-2)m_\theta}}
	{2s_\theta-2}}
	\\
	&=
	C_\theta^{
	\frac{p_{\rm H}-\Delta p\,C_\chi h^{(2s_\theta-2)m_\theta}}
	{2s_\theta-2}}
	\frac{h^{p_{\rm H}}}
	{h^{\Delta p\,C_\chi h^{(2s_\theta-2)m_\theta}}}
	\\
	&=
	C_\theta^{
	\frac{p_{\rm H}-\Delta p\,C_\chi h^{(2s_\theta-2)m_\theta}}
	{2s_\theta-2}}
	h^{p_{\rm H}}
	\exp\!\left(
	\Delta p\,C_\chi h^{(2s_\theta-2)m_\theta}\log(1/h)
	\right).
	\end{aligned}
\]
The last equality uses \(h^{-a}=\exp(-a\log h)=\exp(a\log(1/h))\). Since \((2s_\theta-2)m_\theta>0\), setting \(x=1/h\) and applying L'H\^{o}pital's rule gives
\[
	\begin{aligned}
	\lim_{h\to0^+}
	h^{(2s_\theta-2)m_\theta}\log(1/h)
	&=
	\lim_{x\to\infty}
	\frac{\log x}{x^{(2s_\theta-2)m_\theta}}
	\\
	&=
	\lim_{x\to\infty}
	\frac{1}{(2s_\theta-2)m_\theta x^{(2s_\theta-2)m_\theta}}
	=
	0.
	\end{aligned}
\]
Thus, the exponential factor converges to one, while the power of \(C_\theta\) remains bounded, and hence \(\theta_k^{\beta_k}=\fn{O}(h^{p_{\rm H}})\). Repeating the preceding entropy-budget, element-local, and assembly estimates gives the claimed bounds with \(p_{\rm L}\) replaced by \(p_{\rm H}\).
\end{proof}

\begin{remark}\label{rem:beta_activation_selectivity}
	The low order \(p_{\rm L}\) controls the worst-case consistency based only on \(0\leq\chi_k\leq1\), and may be chosen below \(p\), including \(p_{\rm L}=1\), to strengthen dissipation in active elements. The high order \(p_{\rm H}\) controls resolved smooth elements through the decay of \(\chi_k\). Thus, design-order consistency follows from either \(p_{\rm L}\geq p\) without using the indicator decay or \(p_{\rm H}\geq p\) with the indicator in \cref{eq:theta_exponent_indicator}. These choices do not affect conservation or entropy stability.
\end{remark}
\subsection{Conservation and entropy stability}\label{sec:conservation_entropy_stability}

With the artificial-dissipation operator now fully specified, we formally establish its semidiscrete conservation and entropy-dissipation properties.

\begin{proposition}\label{prop:local_artificial_dissipation_moments}
	The local contribution \(\bm d_k\) in \cref{eq:local_artificial_dissipation_update} with \(\bm \alpha_{k,\epsilon} \) and \(\bm a_{k} \) defined in \eqref{eq:local_alpha_definition} and \eqref{eq:scaled_reference_distribution_a}, respectively, is locally conservative and locally entropy dissipative, \ie,
	\begin{equation}\label{eq:element_artificial_moments}
		\overline{\I}_k^T\bm d_k=0,
		\quad
		\bm w_k^T\bm d_k =-(1-\eta_k)\varepsilon_k\le0,
		\quad
		\eta_k \coloneqq
		\frac{\epsilon_{\rm mach}V_k^2}
		{b_k^2+\epsilon_{\rm mach}V_k^2},
		\quad
		\epsilon_{\rm mach}>0.
	\end{equation}
	Consequently, the assembled artificial dissipation in \cref{eq:local_artificial_dissipation_update} is conservative, \ie, 
	\begin{equation}\label{eq:local_artificial_conservation}
		\overline{\I}_g^T \overline{\H}_g
		\Big(\der[\bm u_g]{t}-\bm r_g\Big)=	0.
	\end{equation}
	Its assembled entropy contribution is
	\begin{equation}\label{eq:local_artificial_entropy}
		\bm w_g^T\overline{\H}_g\der[\bm u_g]{t}
		=
		\bm w_g^T\overline{\H}_g\bm r_g	-\sum_{\Omega_k\in\mathcal{T}_h}(1-\eta_k)\varepsilon_k.
	\end{equation}
	Hence, if the base spatial residual satisfies \(\bm w_g^T\overline{\H}_g\bm r_g\le0\), then
	\begin{equation}\label{eq:local_artificial_entropy_bound}
		\bm w_g^T \overline{\H}_g\der[\bm u_g]{t}\le-\sum_{\Omega_k\in\mathcal{T}_h}(1-\eta_k)\varepsilon_k	\le0 .
	\end{equation}
\end{proposition}

\begin{proof}
	By construction, the distribution coefficient, \(\bm{\alpha}_{k,\epsilon}\), satisfies
    \begin{equation}\label{eq:local_alpha_conditions}
        \overline{\I}_k^T\bm\alpha_{k,\epsilon}=\bm0,
        \qquad
        \bm w_k^T\bm\alpha_{k,\epsilon}=1-\eta_k;
    \end{equation}
    thus it follows that 
	\[
		\overline{\I}_k^T\bm d_k
		=
		-\overline{\I}_k^T \bm\alpha_{k,\epsilon}\varepsilon_k
		=0,
		\qquad
		\bm w_k^T\bm d_k = -(1-\eta_k)\varepsilon_k	\le0.
	\]
	The assembled weak-residual contribution is
	\[
		\overline{\H}_g	\Big(\der[\bm u_g]{t}-\bm r_g\Big)
		=
		\sum_{\Omega_k\in\mathcal{T}_h}	\overline{\P}_k^T\bm d_k.
	\]
	Multiplying this identity by \(\overline{\I}_g^T\) gives
	\[
		\overline{\I}_g^T\overline{\H}_g\Big(\der[\bm u_g]{t}-	\bm r_g\Big)
		=
		\sum_{\Omega_k\in\mathcal{T}_h}	\overline{\I}_k^T\bm d_k = 0.
	\]
	Similarly, we have
	\[
		\bm w_g^T \overline{\H}_g \der[\bm u_g]{t}	
		=
		\bm w_g^T \overline{\H}_g\bm r_g 
		+\sum_{\Omega_k\in\mathcal{T}_h}\bm w_k^T\bm d_k,
	\]
	The identity \(\bm w_k=\overline{\P}_k\bm w_g\), together with the local entropy moment in \cref{eq:element_artificial_moments}, gives \cref{eq:local_artificial_entropy}. The non-positivity bound follows from \(0\le\eta_k\le1\) and \(\varepsilon_k\ge0\).
\end{proof}

The nonnegative scalar \(\varepsilon_k\) is the entropy-dissipation budget assigned to element \(\Omega_k\), whereas the realized entropy removal is \((1-\eta_k)\varepsilon_k\), with \(1-\eta_k=b_k^2/(b_k^2+\epsilon_{\rm mach} V_k^2)\). Thus, \(\eta_k\) is the unapplied fraction and
\(\eta_k\varepsilon_k\) is the unapplied amount of dissipation. If \(\eta_k\approx1\), the entropy-normalized distribution is smoothly attenuated because \(b_k^2\) is small compared with \(\epsilon_{\rm mach} V_k^2\). This can occur for a nearly constant local state or through cancellation in \(b_k=\bm w_k^T\M_k\bm u_k\). If \(b_k=0\), then \(\bm\alpha_{k,\epsilon}=\bm0\) and no dissipation is applied. The regularization keeps the formula finite as the entropy moment approaches zero.

\section{Numerical experiments}\label{sec:numerical_experiments}
The purpose of this section is to study the performance of the dissipation operator, in particular its high-order accuracy, shock capturing capabilities, and robustness for under-resolved vortical flows. Unless stated otherwise, test cases involving the Euler equations use the Hadamard-form two-point flux discretization with the entropy-conserving fluxes of Ranocha \cite{ranocha2018comparison} as the base spatial discretization. The two-dimensional Riemann problem, however, uses a divergence-form discretization of the Euler equations which is not entropy stable, as in the absence of a dedicated positivity-preserving mechanism, positivity violations prevented the use of the entropy-stable scheme. The Burgers test uses the entropy-stable split-form discretization. When an interface-dissipation comparison is reported, we use the local Lax--Friedrichs dissipation. All test degrees satisfy \(p\geq2\). For all test cases, we set \(L_{\rm ref}=1\), \(c_\theta=10^{-4}\), \(m_\theta=4\), \(p_{\rm L}=p\), and \(p_{\rm H}=2p\) in \cref{eq:volume_sensor_operator_normalized_fluctuation,eq:theta_exponent_indicator,eq:local_dissipation_order}. The primary quantities of interest include the solution error, the conservation error, the entropy dissipation, and the convergence of the elementwise entropy budget. Unless stated otherwise, time integration is performed with the adaptive four-stage, third-order strong-stability-preserving Runge--Kutta method (SSPRK43) \cite{fekete2022embedded,kraaijevanger1991contractivity,ranocha2022optimized} from \texttt{OrdinaryDiffEq.jl} \cite{rackauckas2017differentialequations} with absolute and relative tolerances \(10^{-7}\). A computation is terminated if a thermodynamic quantity becomes negative or if the adaptive time step is at most \(10^{-8}\) for ten consecutive time steps. Where it appears, \(\epsilon_{\rm mach}\) denotes the machine precision of the working floating-point type.

In all cases, we use diagonal-\(\E\) SBP operators. In one dimension, we use Legendre--Gauss--Lobatto (LGL) operators. In two dimensions, we use the triangular diagonal-\(\E\) operators with LGL facet nodes from \cite{worku2024quadrature}. In three dimensions, we use tetrahedral operators based on the quadrature rules in \cite{worku2024quadrature}. The triangular and tetrahedral operators use quadrature rules of degree \(2p\). The mesh is constructed by first subdividing each coordinate direction into \(n_{\rm edge}\) segments, then splitting each resulting quadrilateral into two triangles and each hexahedron into six tetrahedra. Thus, the number of elements is \(n_{\rm edge}\) in one dimension, \(2n_{\rm edge}^2\) in two dimensions, and \(6n_{\rm edge}^3\) in three dimensions.

When an exact or reference solution is available, let \(\bm u_g^{\rm ex}(t)\) denote its conservative nodal values of the same size as \(\bm u_g(t)\), and define the error as 
\[
	\bm e_g(t) \coloneqq \bm u_g(t)-\bm u_g^{\rm ex}(t).
\]
For C-SBP discretizations, the error is computed with the assembled block norm matrix defined in \cref{eq:csbp_oper_def}, as
\begin{equation}\label{eq:csbp_error_norm}
	\|\bm e_g\|_{\H}^2 	= \bm e_g^T\overline{\H}_g\bm e_g
	= \sum_{\Omega_k\in\mathcal{T}_h} \bm e_k^T\overline{\H}_k\bm e_k
	=\sum_{\Omega_k\in\mathcal{T}_h} \sum_{\ell=1}^{n_c} \bm e_{k,\ell}^T\H_k\bm e_{k,\ell}.
\end{equation}
For D-SBP discretizations, the same element sum in \cref{eq:csbp_error_norm} is used, with \(\overline{\P}_k\) interpreted as extraction of the \(k\)-th element block. Furthermore, for periodic cases, at each output time, we evaluate the conservation defect as 
\begin{equation}\label{eq:conservation_error}
	\sum_{j=1}^{n_c} \left|	\sum_{\Omega_k\in\mathcal{T}_h}	\bm{1}_k^T\H_k \der[\bm u_{k,j}]{t} \right|.
\end{equation}
Finally, we report the change in entropy as \begin{equation}\label{eq:entropy_metric}
      \frac{\sum_{\Omega_k\in\mathcal{T}_h} \bm{1}_k^T\H_k\bm s_k(t)
	  - \sum_{\Omega_k\in\mathcal{T}_h}\bm{1}_k^T\H_k\bm s_k(0)}{\big|\sum_{\Omega_k\in\mathcal{T}_h}\bm{1}_k^T\H_k\bm s_k(0)\big|}
\end{equation}
when the denominator is nonzero, otherwise only the numerator is calculated.

In tables and figures, we use the notation EC to denote an entropy conservative discretization, AD to denote the artificial dissipation of this work given in \cref{eq:local_artificial_dissipation_update}, and LLF to indicate the use of the local Lax-Friedrichs interface dissipation. 

\subsection{Two- and three-dimensional isentropic vortex}
The isentropic vortex problems \cite{erlebacher1997interaction,williams2013nodal} are used to verify accuracy and convergence rate for smooth solutions. In two dimensions, the domain is \(\Omega=[0,20]\times[-5,5]\) with periodic boundary conditions, and the final time is \(t_f=1\). The exact solution is the moving isentropic vortex
\begin{equation}\label{eq:euler2d_vortex}
	\begin{aligned}
		\rho & = \left(	1- \frac{\alpha^2(\gamma-1)}{16\gamma\pi^2}
		\exp\!\left(2(1-r^2)\right)	\right)^{1/(\gamma-1)},
		\qquad
		p=\rho^\gamma,
		\\
		u& = 1-	\frac{\alpha}{2\pi}(y-y_c)\exp(1-r^2),
		\qquad
		v= \frac{\alpha}{2\pi} \left(x-(x_c+t)\right)\exp(1-r^2),
	\end{aligned}
\end{equation}
where \(r^2=\left(x-(x_c+t)\right)^2+(y-y_c)^2\), \(\alpha=3\), \(\gamma=1.4\), and \((x_c,y_c)=(5,0)\). The initial condition is obtained from the exact solution at \(t=0\). In three dimensions, the domain is \(\Omega=[-10,10]^3\) with periodic boundary conditions and, as in the two-dimensional case, \(t_f=1\). The exact solution is given by
\begin{equation}\label{eq:euler3d_vortex}
	\begin{aligned}
		\rho &=	\left(1-\frac{2}{25}(\gamma-1)	\exp\!\left(1-(y-t)^2-x^2\right)\right)^{1/(\gamma-1)},
		\; 
		e =	\frac{\rho^\gamma}{\gamma(\gamma-1)}+\frac{\rho}{2}(u^2+v^2+w^2),
		\\
		u &=-\frac{2}{5}(y-t) \exp\!\left(\frac{1}{2}\left[1-(y-t)^2-x^2\right]\right),
		\;
		v =	1+\frac{2}{5}x \exp\!\left(\frac{1}{2}\left[1-(y-t)^2-x^2\right]\right),
		\;
		w=0.
	\end{aligned}
\end{equation}
The two- and three-dimensional isentropic-vortex convergence studies use the classical fourth-order Runge--Kutta (RK4) with CFL number \(0.1\). To isolate the entropy-conservative or entropy-dissipative behavior of the spatial discretization from additional entropy changes introduced by time integration, the corresponding long-time studies use the fourth-order relaxation Runge--Kutta (RRK4) \cite{ranocha2020relaxation,ketcheson2019relaxation} method, with a nominal time step based on CFL number \(0.2\) in both dimensions.

\subsubsection{Accuracy}
We conduct convergence studies for both the two- and three-dimensional isentropic vortex problems. \Cref{tab:isentropic-vortex-convergence-2d,tab:isentropic-vortex-convergence-3d} report the \(\H\)-norm solution error and the maximum elementwise entropy budget for the Hadamard-form two-point flux discretizations with and without the proposed artificial dissipation. For C-SBP, the errors with and without AD agree closely and become indistinguishable to the reported digits on the refined meshes, showing that the added dissipation has a negligible effect on the accuracy of the smooth solution. For D-SBP, the errors obtained with LLF and AD are smaller than those of the entropy-conservative discretization; this comparison reflects the combined effect of the interface and volume dissipation. In both two and three dimensions, the solution-error rates approach the expected \(p\)- to \((p+1)\)-order range under refinement. Although not pursued here, rates closer to \(p+1\) can be obtained by interpolating and sampling the solution at higher-degree quadrature nodes to better approximate the \(L^2\) norm \cite{shadpey2020entropy,worku2023entropy}. The maximum entropy budget decreases at an order of at least $2p+d+1$, consistent with the estimate of \eqref{eq:epsilon_convergence} with $p_{\rm L}$ replaced with the sharper $p_{\rm H} = 2p$, confirming that the proposed dissipation is strongly deactivated for the smooth solution and supporting the design-order consistency estimate in \cref{prop:ad_design_order_consistency}.

\begin{table}[!t]
\centering
\scriptsize
\setlength{\tabcolsep}{3.5pt}
\renewcommand{\arraystretch}{1.18}
\caption{Isentropic vortex convergence in 2D at \(t_f=1.0\).}
\label{tab:isentropic-vortex-convergence-2d}
\begin{tabular}{c r cc cccc cc cccc}
\toprule
\addlinespace[0.25em]
 & & \multicolumn{6}{c}{C-SBP} & \multicolumn{6}{c}{D-SBP} \\
\addlinespace[0.15em]
\cmidrule(lr){3-8}\cmidrule(lr){9-14}
\addlinespace[0.15em]
$p$ & $n_x$ & \multicolumn{2}{c}{EC} & \multicolumn{4}{c}{AD \cref{eq:local_artificial_dissipation_update}} & \multicolumn{2}{c}{EC} & \multicolumn{4}{c}{LLF, AD \cref{eq:local_artificial_dissipation_update}} \\
\cmidrule(lr){3-4}\cmidrule(lr){5-8}\cmidrule(lr){9-10}\cmidrule(lr){11-14}
 & & error & $\mathrm{rate}_e$ & error & $\mathrm{rate}_e$ & $\varepsilon_{k,\max}$ & $\mathrm{rate}_{\varepsilon}$ & error & $\mathrm{rate}_e$ & error & $\mathrm{rate}_e$ & $\varepsilon_{k,\max}$ & $\mathrm{rate}_{\varepsilon}$ \\
\addlinespace[0.15em]
\midrule
\addlinespace[0.1em]
3 & 32 & 1.07e-02 & -- & 1.07e-02 & -- & 6.52e-06 & -- & 8.25e-03 & -- & 4.68e-03 & -- & 1.02e-08 & -- \\
 & 64 & 9.27e-04 & 3.53 & 9.27e-04 & 3.52 & 6.84e-11 & 16.54 & 7.05e-04 & 3.55 & 3.81e-04 & 3.62 & 1.54e-11 & 9.36 \\
 & 128 & 9.30e-05 & 3.32 & 9.30e-05 & 3.32 & 2.21e-13 & 8.27 & 6.01e-05 & 3.55 & 3.04e-05 & 3.65 & 3.16e-14 & 8.93 \\
 & 256 & 7.12e-06 & 3.71 & 7.12e-06 & 3.71 & 1.45e-16 & 10.57 & 6.39e-06 & 3.23 & 2.25e-06 & 3.75 & 5.67e-17 & 9.12 \\
\addlinespace[0.35em]
4 & 32 & 1.55e-03 & -- & 1.55e-03 & -- & 3.23e-13 & -- & 1.43e-03 & -- & 6.72e-04 & -- & 9.76e-14 & -- \\
 & 64 & 8.38e-05 & 4.21 & 8.38e-05 & 4.21 & 9.05e-17 & 11.80 & 7.32e-05 & 4.28 & 3.05e-05 & 4.46 & 7.74e-17 & 10.30 \\
 & 128 & 4.73e-06 & 4.15 & 4.73e-06 & 4.15 & 5.49e-20 & 10.68 & 3.63e-06 & 4.33 & 1.31e-06 & 4.55 & 5.32e-20 & 10.51 \\
 & 256 & 2.14e-07 & 4.46 & 2.14e-07 & 4.46 & 2.63e-23 & 11.03 & 1.69e-07 & 4.43 & 5.03e-08 & 4.70 & 2.55e-23 & 11.03 \\
\addlinespace[0.1em]
\bottomrule
\end{tabular}
\end{table}

\begin{table}[!t]
\centering
\scriptsize
\setlength{\tabcolsep}{3.5pt}
\renewcommand{\arraystretch}{1.18}
\caption{Isentropic vortex convergence in 3D at \(t_f=1.0\).}
\label{tab:isentropic-vortex-convergence-3d}
\begin{tabular}{c r cc cccc cc cccc}
\toprule
\addlinespace[0.25em]
 & & \multicolumn{6}{c}{C-SBP} & \multicolumn{6}{c}{D-SBP} \\
\addlinespace[0.15em]
\cmidrule(lr){3-8}\cmidrule(lr){9-14}
\addlinespace[0.15em]
$p$ & $n_x$ & \multicolumn{2}{c}{EC} & \multicolumn{4}{c}{AD \cref{eq:local_artificial_dissipation_update}} & \multicolumn{2}{c}{EC} & \multicolumn{4}{c}{LLF, AD \cref{eq:local_artificial_dissipation_update}} \\
\cmidrule(lr){3-4}\cmidrule(lr){5-8}\cmidrule(lr){9-10}\cmidrule(lr){11-14}
 & & error & $\mathrm{rate}_e$ & error & $\mathrm{rate}_e$ & $\varepsilon_{k,\max}$ & $\mathrm{rate}_{\varepsilon}$ & error & $\mathrm{rate}_e$ & error & $\mathrm{rate}_e$ & $\varepsilon_{k,\max}$ & $\mathrm{rate}_{\varepsilon}$ \\
\addlinespace[0.15em]
\midrule
\addlinespace[0.1em]
3 & 10 & 6.01e-01 & -- & 5.93e-01 & -- & 2.12e-03 & -- & 5.35e-01 & -- & 3.75e-01 & -- & 2.19e-04 & -- \\
 & 15 & 2.24e-01 & 2.43 & 2.24e-01 & 2.40 & 4.34e-05 & 9.59 & 1.74e-01 & 2.77 & 1.09e-01 & 3.05 & 1.39e-08 & 23.84 \\
 & 20 & 1.01e-01 & 2.76 & 1.01e-01 & 2.75 & 6.00e-08 & 22.89 & 7.61e-02 & 2.88 & 4.51e-02 & 3.07 & 1.81e-09 & 7.07 \\
 & 25 & 5.00e-02 & 3.16 & 5.00e-02 & 3.16 & 2.62e-09 & 14.03 & 3.79e-02 & 3.13 & 2.08e-02 & 3.48 & 1.95e-10 & 10.00 \\
\addlinespace[0.35em]
4 & 10 & 2.46e-01 & -- & 2.46e-01 & -- & 1.57e-09 & -- & 2.29e-01 & -- & 1.32e-01 & -- & 1.60e-10 & -- \\
 & 15 & 5.92e-02 & 3.52 & 5.92e-02 & 3.52 & 6.56e-12 & 13.50 & 6.12e-02 & 3.25 & 3.24e-02 & 3.47 & 7.15e-13 & 13.35 \\
 & 20 & 1.77e-02 & 4.20 & 1.77e-02 & 4.20 & 9.09e-14 & 14.87 & 1.86e-02 & 4.14 & 9.57e-03 & 4.24 & 8.16e-15 & 15.55 \\
 & 25 & 7.36e-03 & 3.93 & 7.36e-03 & 3.93 & 3.34e-15 & 14.81 & 7.46e-03 & 4.09 & 3.77e-03 & 4.17 & 2.47e-16 & 15.67 \\
\addlinespace[0.1em]
\bottomrule
\end{tabular}
\end{table}

\subsubsection{Conservation, entropy stability, and long-time error behavior}
We also use the isentropic vortex problems to study conservation, entropy dissipation, and long-time error growth. For both problems, the solutions are evolved to \(t_f=80\). As pointed out in \cite{spiegel2015survey}, long-time isentropic vortex analyses on periodic domains should be interpreted with care, as periodic isentropic-vortex calculations with high-order and low-dissipation methods can recirculate and accumulate error. Here, the test is used as a controlled comparison, with all methods run using the same mesh, operator family, time integrator, and initial condition. Thus, differences in the long-time outputs of interest reflect the relative effect of the added dissipation in this setting.

\begin{figure}[!t]
	\centering
	\begin{subfigure}{0.48\textwidth}
		\includegraphics[width=\linewidth]{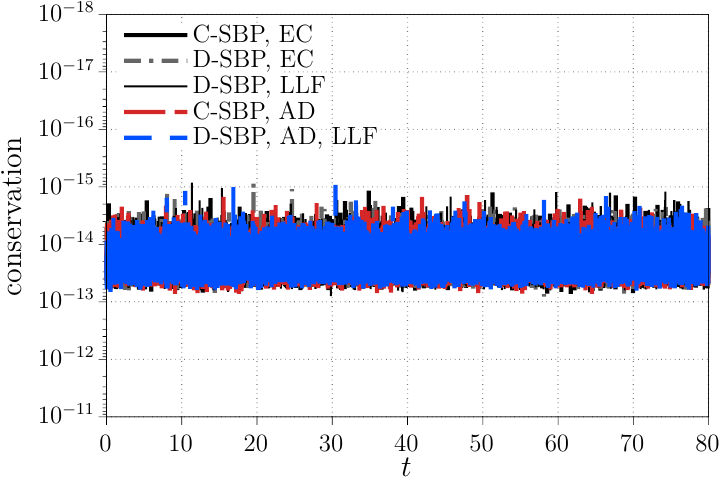}
		\caption{2D, \(p=4\), \(n_{\rm edge}=32\).}
	\end{subfigure}
	\hfill
	\begin{subfigure}{0.48\textwidth}
		\includegraphics[width=\linewidth]{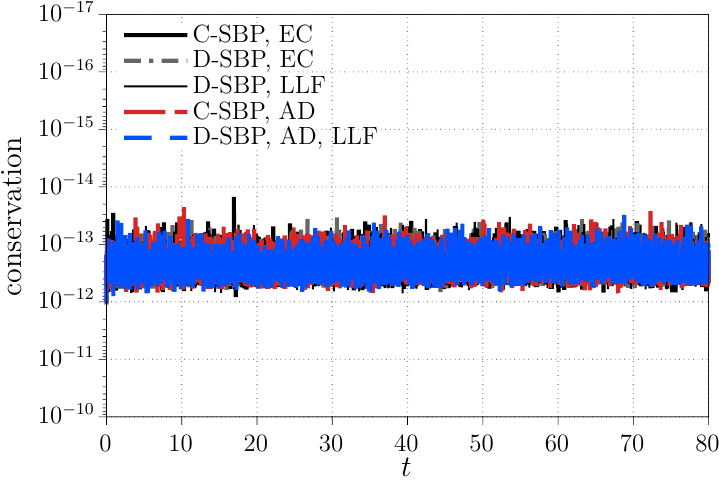}
		\caption{3D, \(p=4\), \(n_{\rm edge}=10\).}
	\end{subfigure}
	\caption{Conservation error for the long-time isentropic vortex tests.}
	\label{fig:isentropic-vortex-conservation}
\end{figure}

Supporting the result in \cref{prop:local_artificial_dissipation_moments}, \cref{fig:isentropic-vortex-conservation} shows that the proposed artificial dissipation preserves conservation to machine precision in both two and three dimensions. The change in mathematical entropy, defined in \cref{eq:entropy_metric}, is shown in \cref{fig:isentropic-vortex-entropy}. With RRK4, the entropy-conservative baselines exhibit changes close to zero, while the computations with AD, LLF, or both exhibit nonpositive entropy changes, consistent with the semidiscrete entropy properties. In two dimensions, the small EC entropy drift becomes more pronounced near the periodic boundary crossings of the vortex, at approximately \(t=15,35,55,75\). The isolated-vortex profile is not exactly periodic on a finite domain, and the boundary artifacts and recirculation of numerical errors discussed in \cite{spiegel2015survey} may contribute to the observed behavior. In both two and three dimensions, the D-SBP histories with LLF alone and with LLF and AD nearly coincide, indicating that the added volume dissipation has a negligible additional effect on the entropy change for this smooth problem.
\begin{figure}[!t]
	\centering
	\begin{subfigure}{0.48\textwidth}
		\includegraphics[width=\linewidth]{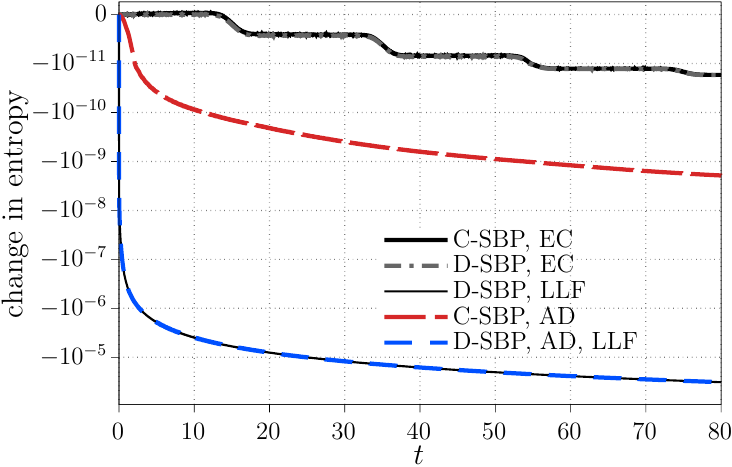}
		\caption{2D, \(p=4\), \(n_{\rm edge}=32\).}
	\end{subfigure}
	\hfill
	\begin{subfigure}{0.48\textwidth}
		\includegraphics[width=\linewidth]{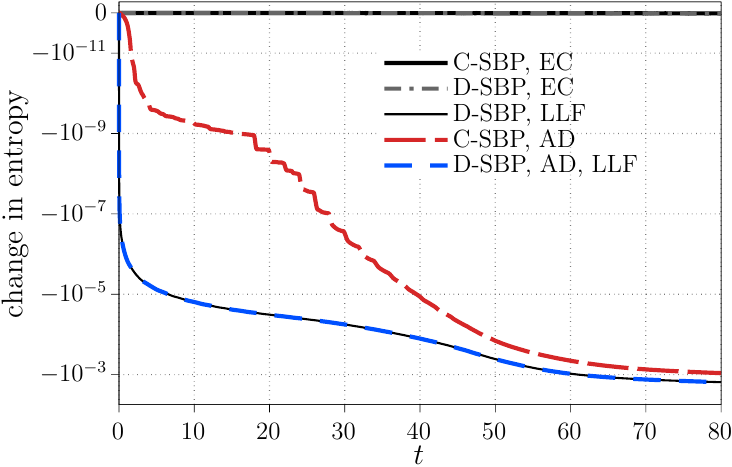}
		\caption{3D, \(p=4\), \(n_{\rm edge}=10\).}
	\end{subfigure}
	\caption{Change in entropy, \cref{eq:entropy_metric}, for the long-time isentropic vortex runs.}
	\label{fig:isentropic-vortex-entropy}
\end{figure}

\Cref{fig:isentropic-vortex-error} compares the long-time solution-error histories. In two dimensions, the C-SBP results with and without AD are indistinguishable, as are the D-SBP results with LLF and with LLF and AD. In three dimensions, the two D-SBP histories with LLF also nearly coincide, while C-SBP with AD follows the entropy-conservative result closely during the early evolution and produces a moderately smaller error at later times. Thus, the proposed volume dissipation does not significantly increase the long-time error for these smooth problems. The substantially smaller D-SBP errors relative to the entropy-conservative discretization are primarily associated with the LLF interface dissipation, since adding AD to LLF produces essentially no further change.
\begin{figure}[!t]
	\centering
	\begin{subfigure}{0.48\textwidth}
		\includegraphics[width=\linewidth]{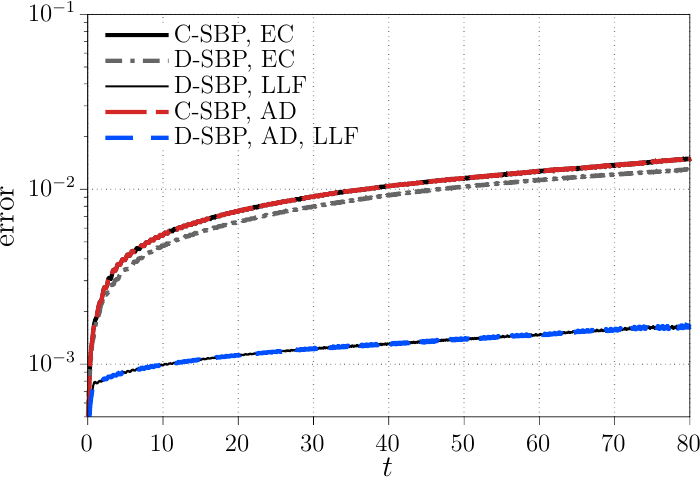}
		\caption{2D, \(n_{\rm edge}=32\).}
		\label{fig:isentropic-vortex-error-coarse}
	\end{subfigure}
	\hfill
	\begin{subfigure}{0.48\textwidth}
		\includegraphics[width=\linewidth]{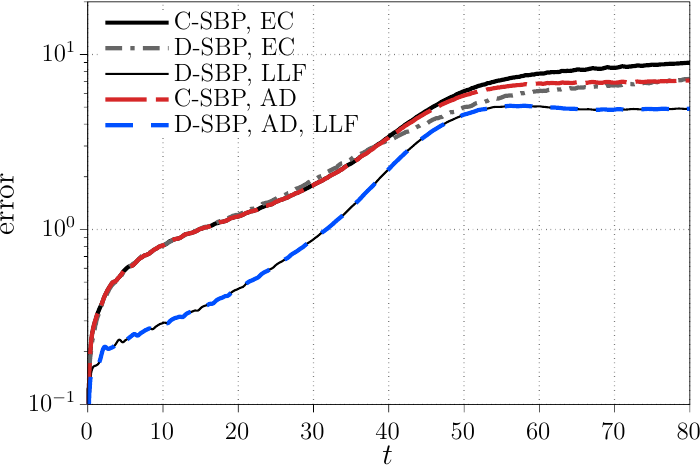}
		\caption{3D, \(n_{\rm edge}=10\).}
		\label{fig:isentropic-vortex-error-3d}
	\end{subfigure}
	\caption{Long-time error histories for the \(p=4\) isentropic vortex tests.}
	\label{fig:isentropic-vortex-error}
\end{figure}
Finally, \cref{fig:isentropic-vortex-epsilon} shows the corresponding total entropy budgets, \(\varepsilon_h=\sum_{\Omega_k\in\mathcal{T}_h}\varepsilon_k\). In 2D, the budgets remain small throughout the simulation, and refinement from \(n_{\rm edge}=32\) to \(n_{\rm edge}=64\) reduces their final values by approximately four and three orders of magnitude for C-SBP and D-SBP, respectively. In 3D, the budgets are strongly time dependent, increasing by several orders of magnitude before reaching a maximum and subsequently decreasing, with C-SBP producing the larger budget.
\begin{figure}[!t]
	\centering
	\begin{subfigure}{0.48\textwidth}
		\includegraphics[width=\linewidth]{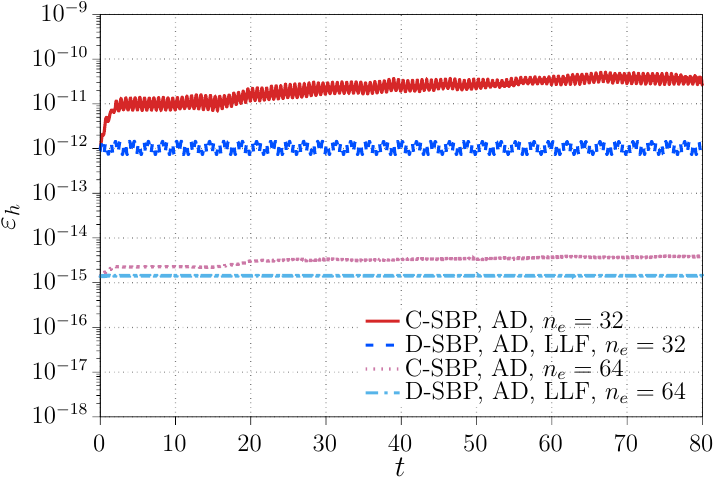}
		\caption{2D, \(p=4\), \(n_{\rm edge}=32,64\).}
	\end{subfigure}
	\hfill
	\begin{subfigure}{0.48\textwidth}
		\includegraphics[width=\linewidth]{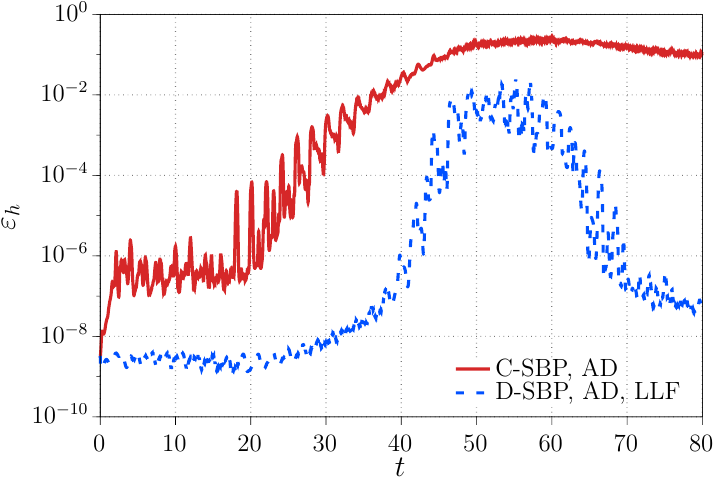}
		\caption{3D, \(p=4\), \(n_{\rm edge}=10\).}
	\end{subfigure}
	\caption{Total entropy budget in the long-time isentropic
	vortex tests with the proposed artificial dissipation.}
	\label{fig:isentropic-vortex-epsilon}
\end{figure}

\subsection{One-dimensional Burgers shock problem}
The inviscid Burgers equation is used as the first test involving shock waves for the proposed artificial dissipation. The equation in the split form is written as
\begin{equation}\label{eq:burgers_pde}
	\pder[\fn{U}]{t} + \frac{1}{3}\pder[\fn{U}^{2}]{x}
	+ \frac{1}{3}\fn{U}\pder[\fn{U}]{x}	=0,
	\qquad	\Omega=[0,1],
\end{equation}
with periodic boundary conditions and initial condition
\begin{equation}\label{eq:burgers_ic}
	\fn{U}(x,0)=2+\sin(2\pi x).
\end{equation}

For this scalar problem, the entropy-stable split-form discretization with SATs is used as the base spatial discretization. This discretization is equivalent to the Hadamard-form discretization, and it is included here for completeness:
\[
	\der[\bm u_k]{t} =	-\frac{1}{3}\D_{x k}(\bm u_k\circ\bm u_k)
	-\frac{1}{3}\bm u_k\circ\D_{x k}\bm u_k
	+ \sum_{\gamma\subset\Gamma_k} \H_k^{-1}\R_{\gamma k}^{T}\B_{\gamma}\bm g_{\gamma k},
\]
where, for \(\gamma=\Gamma_k \cap \Gamma_v\), \(\bm u_{\gamma_k}=\R_{\gamma k}\bm u_k\) and \(\bm u_{\gamma_v}=\R_{\gamma v}\bm u_v\) are the two facet traces, and
\begin{align*}
	\bm g_{\gamma k} &= \N_{x\gamma k} \left(\frac{1}{3}\R_{\gamma k}(\bm u_k\circ\bm u_k)
	- \frac{1}{6}\bm u_{\gamma_k}\circ\bm u_{\gamma_v}
	- \frac{1}{6}\bm u_{\gamma_v}\circ\bm u_{\gamma_v}\right)
	\\
	&\quad - \frac{\phi}{4}\left(\left|\N_{x\gamma k}\bm u_{\gamma_k}\right|
	+\left|\N_{x\gamma v}\bm u_{\gamma_v}\right|\right)	\circ
	(\bm u_{\gamma_k}-\bm u_{\gamma_v}).
\end{align*}
The \(\phi\) parameter controls the interface dissipation; \(\phi=1\) gives upwind fluxes while \(\phi=0\) gives the entropy-conservative SAT.

We run the Burgers problem above up to \(t_f=1\) and compare the numerical solution to the exact solution. \Cref{fig:burgers-moving-shock} shows the \(p=4\), \(n_{\rm edge}=50\) results. Both the C-SBP and D-SBP discretizations with the proposed artificial dissipation capture the shock at the correct location. On this mesh, both the C-SBP and D-SBP results have localized oscillations near the discontinuity, and the D-SBP result with artificial dissipation and upwind SATs is close to the upwind-only result. The lower panels show the elementwise smoothness sensor and entropy budget. The sensor is largest near the shock and can show small background responses, while the entropy budget \(\varepsilon_k\) is sharply localized at the shock, as intended. We stress that this dissipation is not designed to be a rigorous shock-capturing mechanism, and therefore the oscillations localized at the shock should not be interpreted as a failure of the dissipation operator. Although not plotted, we have confirmed that the discretization is conservative and entropy dissipative. Furthermore, additional comparisons using matched adaptive time-integration settings showed no substantial increase in the number of time steps required to reach the final time.
\begin{figure}[!t]
	\centering
	\begin{subfigure}{0.48\textwidth}
		\includegraphics[width=\linewidth]{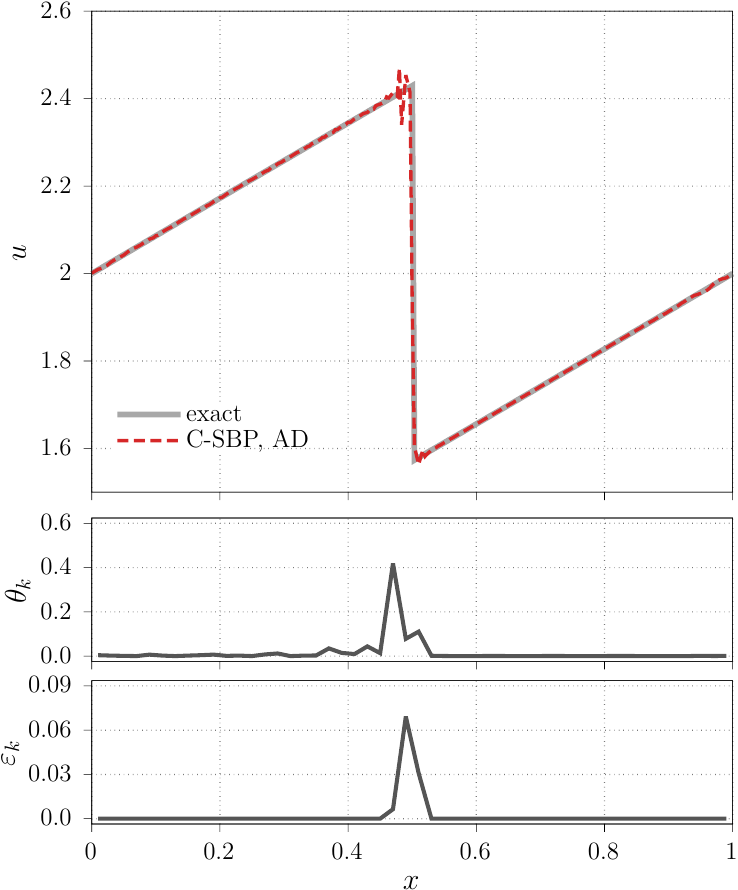}
		\caption{C-SBP discretization.}
	\end{subfigure}
	\hfill
	\begin{subfigure}{0.48\textwidth}
		\includegraphics[width=\linewidth]{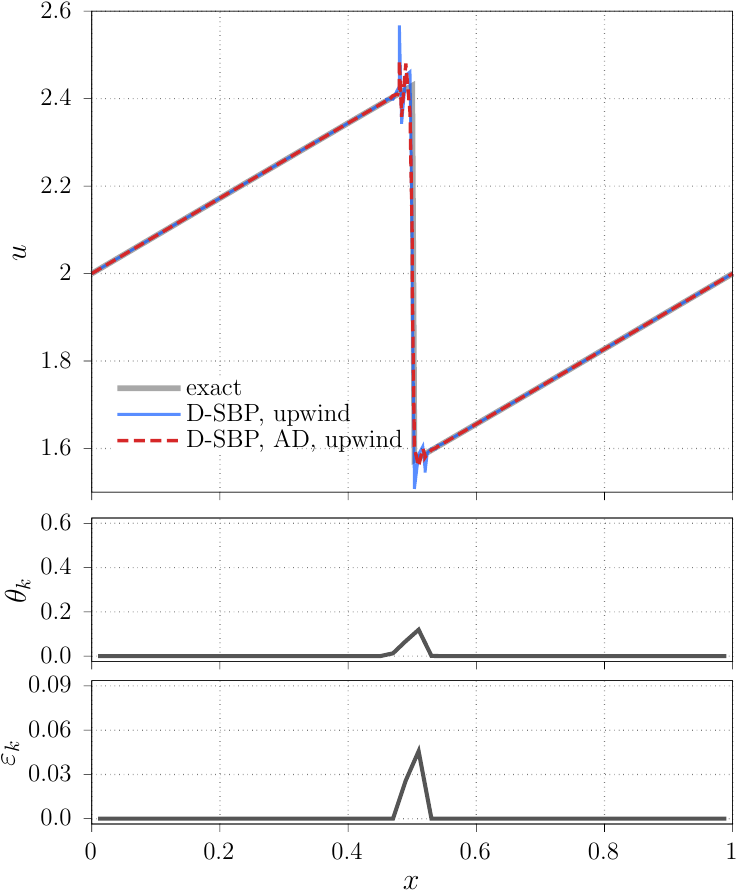}
		\caption{D-SBP discretizations.}
	\end{subfigure}
	\caption{Burgers moving shock at \(t_f=1\), using \(p=4\) and
	\(n_{\rm edge}=50\).}
	\label{fig:burgers-moving-shock}
\end{figure}

\subsection{Sod shock tube}
The Sod shock tube problem \cite{sod1978survey} provides a simple Euler shock, contact, and rarefaction with an exact Riemann solution. We solve the one-dimensional Euler equations on \(\Omega=[0,1]\) until \(t_f=0.2\) with initial data
\begin{equation}\label{eq:sod_ic}
	(\rho,u,p)(x,0)
	=
	\begin{cases}
		(1,0,1),       & x\le 0.5, \\
		(0.125,0,0.1), & x>0.5 .
	\end{cases}
\end{equation}
The exact solution is computed using the iterative shock-tube construction described in \cite[Sec.~3.3.2]{pulliam2014fundamental}.

\Cref{fig:sod-shock-tube} shows the density for \(p=4\), using \(n_{\rm edge}=100\). With the AD, both C-SBP and D-SBP capture the rarefaction, contact, and shock with minimal oscillations and remain close to the exact Riemann solution. The smoothness sensor \(\theta_k\) is not localized only at the shock; it can also respond to small oscillations or roughness away from the shock. This behavior is consistent with the relative-smoothness interpretation in \cref{rem:relative_roughness_sensor}. The entropy budget \(\varepsilon_k\), however, is concentrated on the elements near the shock, so the added dissipation is targeted. For D-SBP, compared with the split-form discretization using LLF interface dissipation alone, the proposed artificial dissipation reduces oscillations near the discontinuities, and the C-SBP result shows that the same residual-correction-based dissipation also captures the shock accurately despite the continuous nature of the discrete solution.
\begin{figure}[!t]
	\centering
	\begin{subfigure}{0.48\textwidth}
		\includegraphics[width=\linewidth]{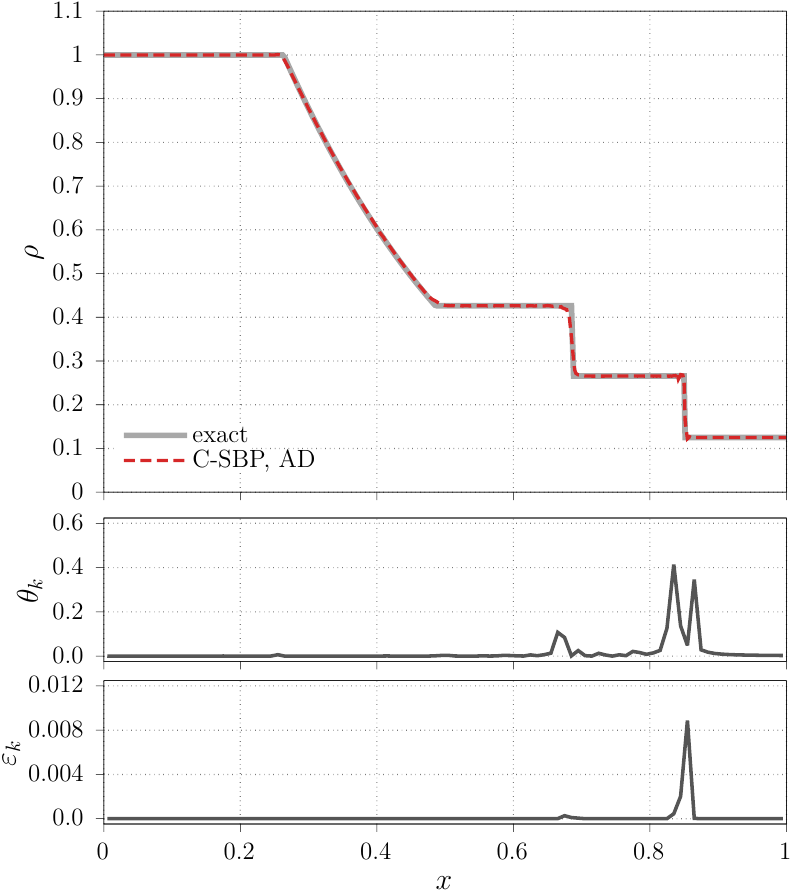}
		\caption{C-SBP, \(p=4\).}
	\end{subfigure}
	\hfill
	\begin{subfigure}{0.48\textwidth}
		\includegraphics[width=\linewidth]{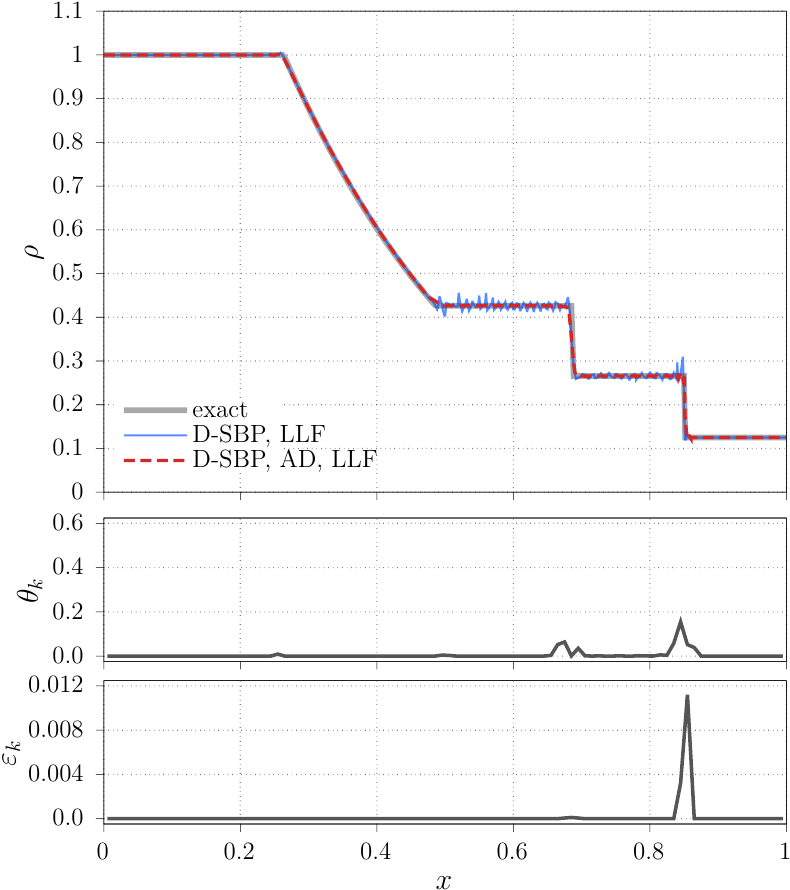}
		\caption{D-SBP, \(p=4\).}
	\end{subfigure}
	\caption{Sod shock tube density at \(t_f=0.2\), using
	\(p=4\) and \(n_{\rm edge}=100\). The lower panels in each plot show the
	elementwise smoothness sensor \(\theta_k\) and entropy budget
	\(\varepsilon_k\).}
	\label{fig:sod-shock-tube}
\end{figure}

\Cref{fig:sod-shock-tube-entropy-change} compares the normalized change in mathematical entropy for \(p\in\{3,4\}\) with that of the exact Riemann solution. The entropy change is negative for all three discretizations. The C-SBP and D-SBP solutions with AD exhibit similar entropy histories and produce a larger reduction in mathematical entropy than both the exact solution and the D-SBP scheme with LLF dissipation alone. This excess entropy reduction is more pronounced for \(p=3\) than for \(p=4\), consistent with the greater dissipation expected at lower polynomial degree. Interestingly, the D-SBP scheme with LLF and AD produces a smaller entropy reduction than the C-SBP scheme with AD alone. One possible explanation is that the LLF interface dissipation suppresses some of the oscillatory content that would otherwise activate the volume artificial dissipation. The D-SBP scheme with LLF dissipation alone follows the exact entropy change most closely; its small amount of dissipation explains the larger pointwise solution oscillations observed in \cref{fig:sod-shock-tube}. 

\begin{figure}[!t]
	\centering
	\begin{subfigure}{0.48\textwidth}
		\includegraphics[width=\linewidth]{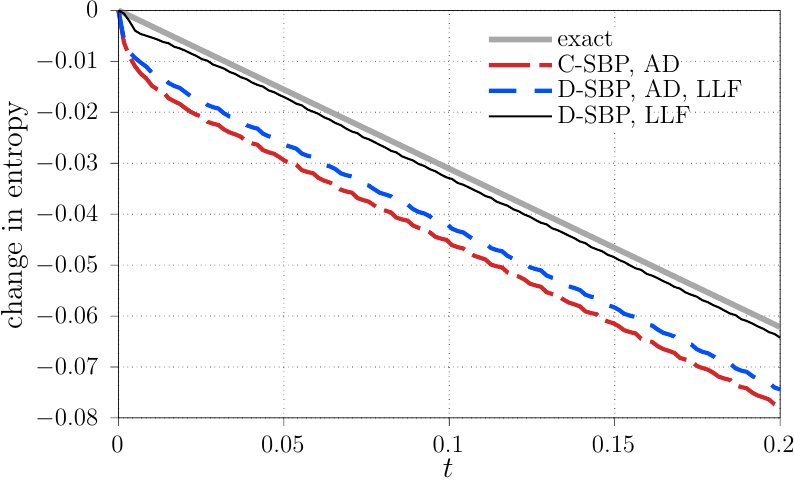}
		\caption{\(p=3\).}
	\end{subfigure}
	\hfill
	\begin{subfigure}{0.48\textwidth}
		\includegraphics[width=\linewidth]{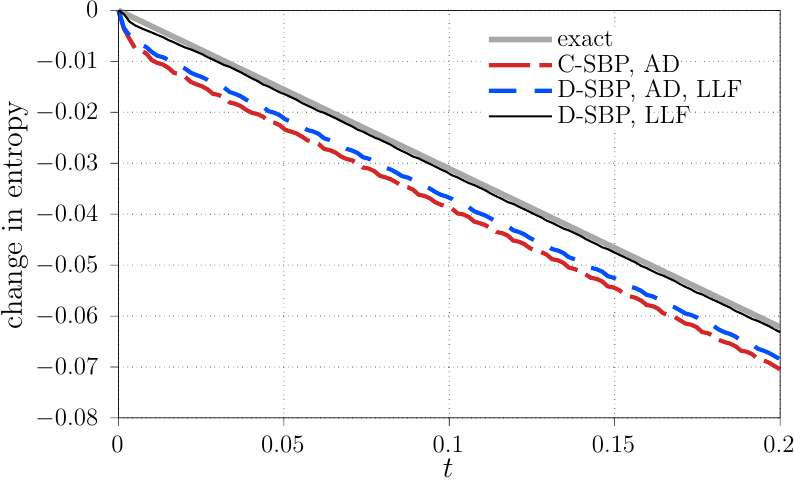}
		\caption{\(p=4\).}
	\end{subfigure}
	\caption{Change in entropy, \cref{eq:entropy_metric}, for the Sod shock tube problem with \(n_{\rm edge}=100\).}
	\label{fig:sod-shock-tube-entropy-change}
\end{figure}

\subsection{Shu--Osher shock tube}
The Shu--Osher problem \cite{shu1989efficient} tests whether the proposed dissipation operator can enable a robust solution for a moving shock by reducing the amplitude of oscillations and thus positivity violations without overdamping smooth post-shock oscillations. The one-dimensional Euler equations are solved on \(\Omega=[-5,5]\) until \(t_f=1.8\) with the initial data
\begin{equation}\label{eq:shu_osher_ic}
	(\rho,u,p)(x,0)	=
	\begin{cases}
		(3.857143,\;2.629369,\;10.33333), & x\le -4, \\
		(1+0.2\sin(5x),\;0,\;1),          & x>-4 .
	\end{cases}
\end{equation}
The reference solution is taken from Chan \cite{chan2025artificial} and was computed using a fifth-order weighted essentially non-oscillatory (WENO) scheme with \(25000\) cells. \Cref{fig:shu-osher-shock-tube} shows the density profile obtained using \(p=4\) and \(p=8\) operators on a mesh with \(n_{\rm edge}=100\). Both formulations with the proposed dissipation reach \(t_f=1.8\). At fixed \(n_{\rm edge}=100\), increasing the degree from \(p=4\) to \(p=8\) improves agreement with the reference post-shock density oscillations, showing that the dissipation retains the benefit of increased polynomial degree. At \(p=4\), the D-SBP split-form discretization with LLF interface dissipation alone reaches the final time, but develops much larger oscillations than when used with the proposed dissipation. At \(p=8\), the D-SBP discretization with LLF dissipation alone is terminated as its adaptive time step is at most \(10^{-8}\) for ten consecutive steps. The lower panels in \cref{fig:shu-osher-shock-tube} show that \(\varepsilon_k\) is concentrated near the moving shock, while \(\theta_k\), as expected from \cref{rem:relative_roughness_sensor}, also responds to smaller oscillations in the post-shock region.
\begin{figure}[!t]
	\centering
	\begin{subfigure}{0.48\textwidth}
		\includegraphics[width=\linewidth]{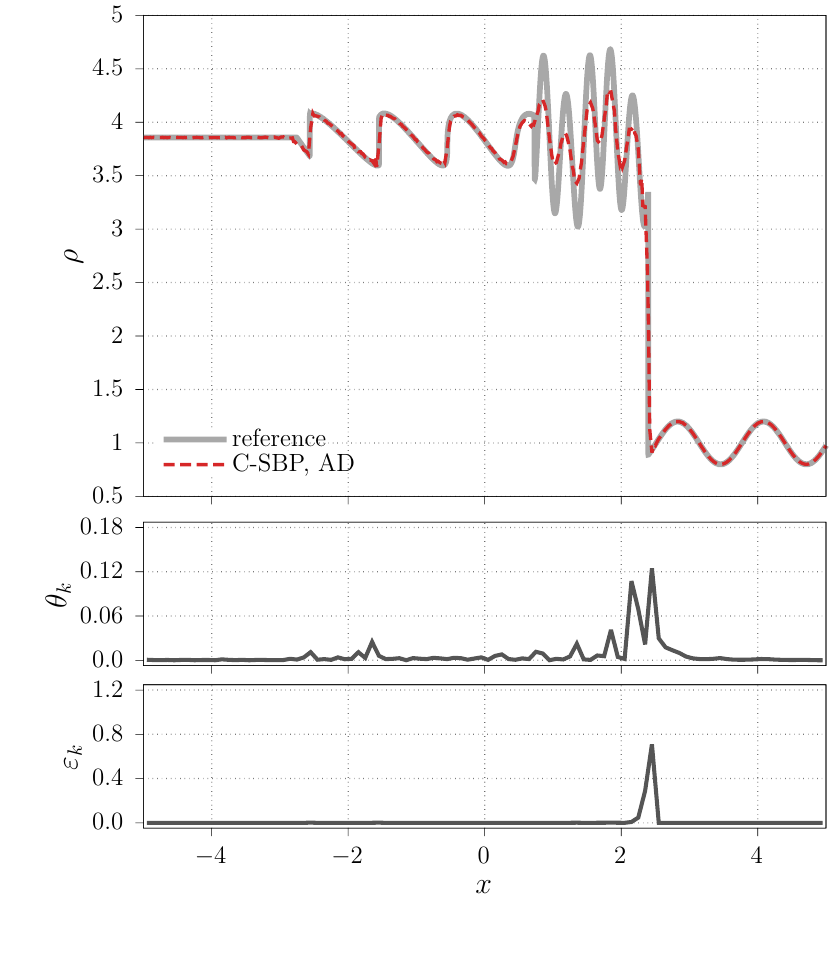}
		\caption{C-SBP, \(p=4\).}
	\end{subfigure}
	\hfill
	\begin{subfigure}{0.48\textwidth}
		\includegraphics[width=\linewidth]{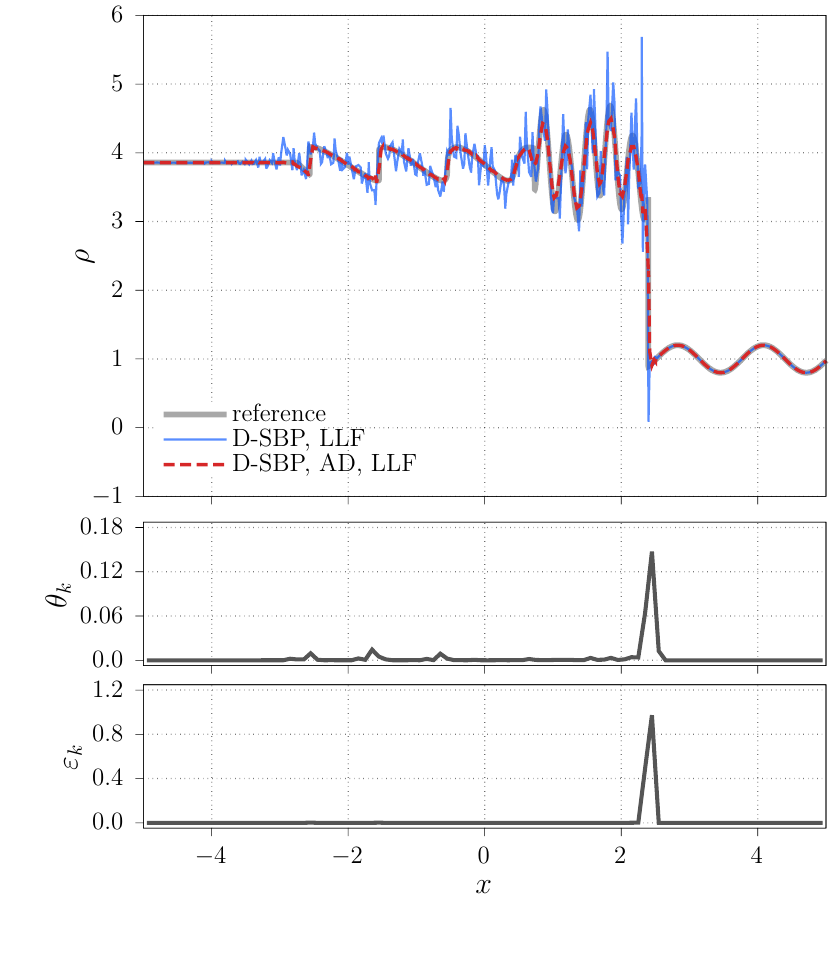}
		\caption{D-SBP, \(p=4\).}
	\end{subfigure}
	\begin{subfigure}{0.48\textwidth}
		\includegraphics[width=\linewidth]{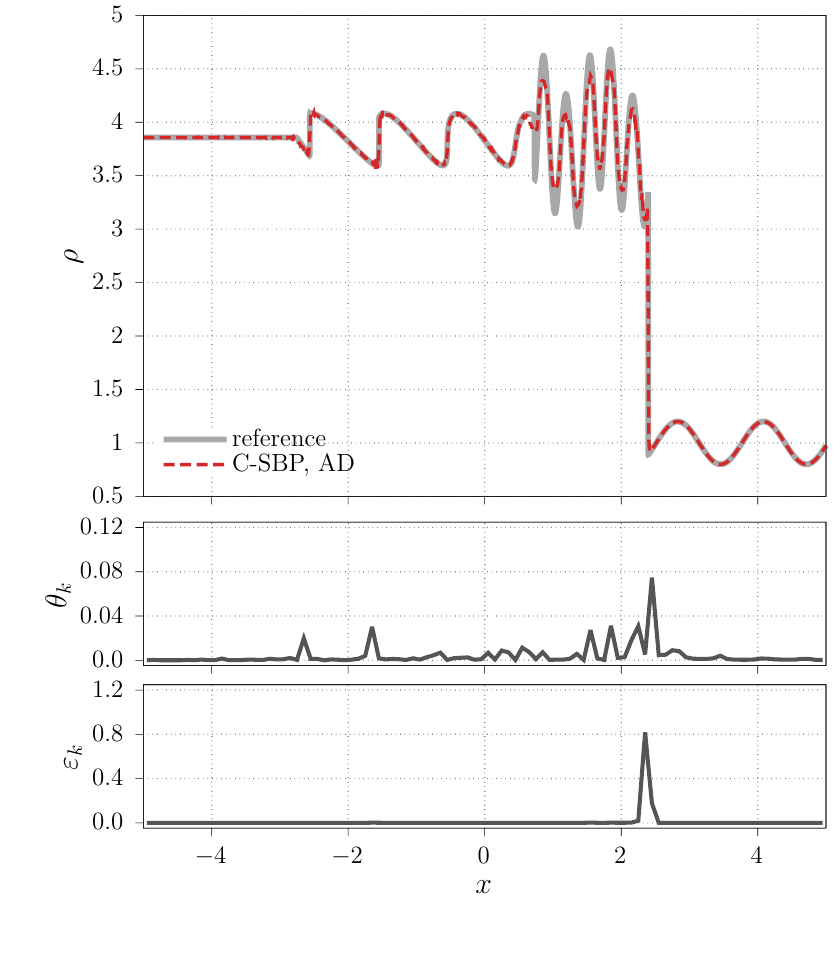}
		\caption{C-SBP, \(p=8\).}
	\end{subfigure}
	\hfill
	\begin{subfigure}{0.48\textwidth}
		\includegraphics[width=\linewidth]{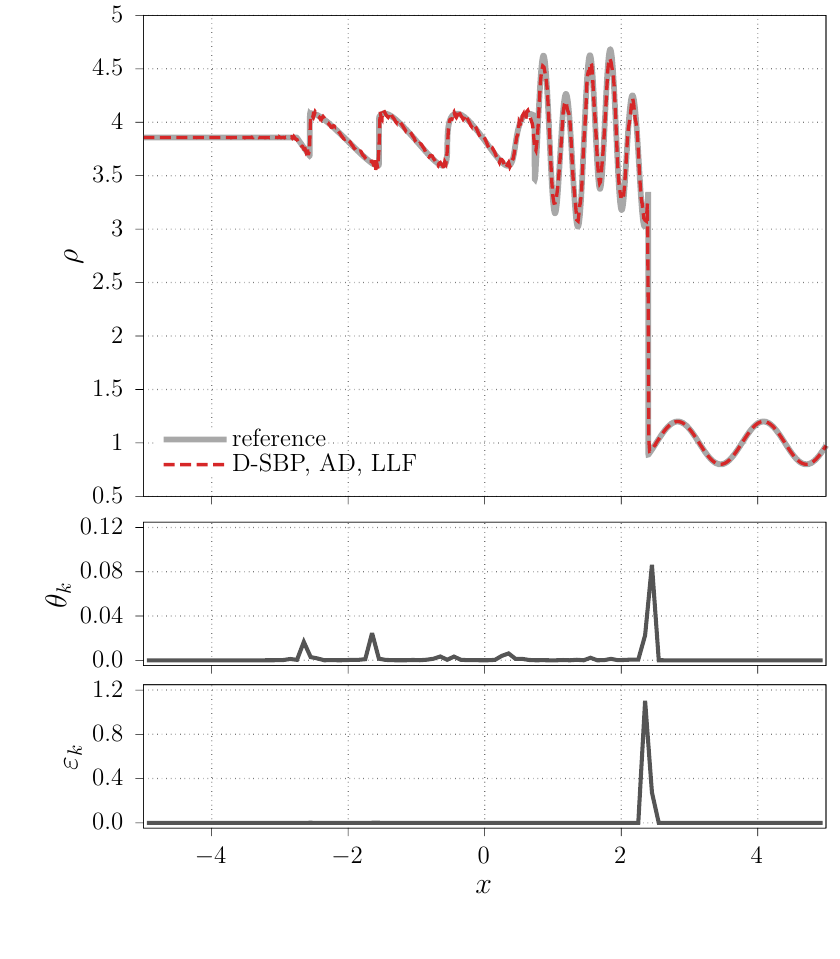}
		\caption{D-SBP, \(p=8\).}
	\end{subfigure}
	\caption{Shu--Osher density at \(t_f=1.8\), using
	\(n_{\rm edge}=100\).}
	\label{fig:shu-osher-shock-tube}
\end{figure}

\subsection{Kelvin--Helmholtz instability}
The Kelvin--Helmholtz problem is included to test whether the operator-normalized volume sensor introduces excessive dissipation in under-resolved vortical flows and whether the artificial dissipation in \cref{eq:local_artificial_dissipation_update} improves robustness, as this problem is known to be difficult to solve using high-order methods. The two-dimensional Euler equations are solved on \(\Omega=[-1,1]^2\) until \(t_f=15\) with periodic boundary conditions. Let
\[
	B(y)=\tanh(15y+7.5)-\tanh(15y-7.5),
\]
then, following the setup in \cite{rueda2022subcell}, the primitive initial condition is given by
\begin{equation}\label{eq:khi_ic}
	\rho=\frac{1}{2}+\frac{3}{4}B(y),
	\qquad
	u=\frac{1}{2}\left(B(y)-1\right),
	\qquad
	v=\frac{1}{10}\sin(2\pi x),
	\qquad
	p=1 .
\end{equation}

\Cref{tab:kh-failure-combined} reports the failure times for \(n_{\rm edge}=8,16,32\) and polynomial degrees \(p\in \{2,\ldots,10\}\). An entry of \(15.00\) means that the run reached the final time. Every split-form D-SBP computation with LLF interface dissipation alone fails by \(t=5.34\). With the proposed artificial dissipation, all C-SBP and D-SBP computations with \(p\leq7\) reach \(t_f=15\). Overall, \(21\) of the \(27\) cases reach the final time for each formulation, with the remaining failures confined to \(p\geq8\). 

\begin{table}[!t]
\centering
\scriptsize
\setlength{\tabcolsep}{4.5pt}
\renewcommand{\arraystretch}{1.15}
\caption{Kelvin--Helmholtz instability failure times. Bold entries did not reach $t_f=15.0$.}
\label{tab:kh-failure-combined}
\begin{tabular}{l r@{\hspace{0.9em}} c c c c c c c c c}
\toprule
Scheme & $n_{\rm edge}$ & $p=2$ & $p=3$ & $p=4$ & $p=5$ & $p=6$ & $p=7$ & $p=8$ & $p=9$ & $p=10$ \\
\midrule
\multirow{3}{*}{D-SBP, LLF} & 8 & \textbf{5.34} & \textbf{4.29} & \textbf{4.26} & \textbf{2.31} & \textbf{3.73} & \textbf{3.29} & \textbf{1.95} & \textbf{2.07} & \textbf{2.14} \\
 & 16 & \textbf{5.19} & \textbf{3.49} & \textbf{1.60} & \textbf{3.51} & \textbf{2.29} & \textbf{3.74} & \textbf{3.57} & \textbf{2.40} & \textbf{3.57} \\
 & 32 & \textbf{3.30} & \textbf{3.74} & \textbf{3.69} & \textbf{3.04} & \textbf{3.61} & \textbf{3.11} & \textbf{3.13} & \textbf{3.18} & \textbf{3.07} \\
\addlinespace[0.6em]
\multirow{3}{*}{D-SBP, AD, LLF} & 8 & 15.00 & 15.00 & 15.00 & 15.00 & 15.00 & 15.00 & 15.00 & \textbf{3.27} & 15.00 \\
 & 16 & 15.00 & 15.00 & 15.00 & 15.00 & 15.00 & 15.00 & 15.00 & \textbf{8.13} & \textbf{12.88} \\
 & 32 & 15.00 & 15.00 & 15.00 & 15.00 & 15.00 & 15.00 & \textbf{8.42} & \textbf{12.00} & \textbf{10.39} \\
\addlinespace[0.6em]
\multirow{3}{*}{C-SBP, AD} & 8 & 15.00 & 15.00 & 15.00 & 15.00 & 15.00 & 15.00 & 15.00 & \textbf{3.35} & \textbf{3.71} \\
 & 16 & 15.00 & 15.00 & 15.00 & 15.00 & 15.00 & 15.00 & 15.00 & 15.00 & \textbf{4.82} \\
 & 32 & 15.00 & 15.00 & 15.00 & 15.00 & 15.00 & 15.00 & \textbf{13.51} & \textbf{4.23} & \textbf{5.01} \\
\bottomrule
\end{tabular}
\end{table}

\Cref{fig:kh-csbp-diagnostics,fig:kh-dsbp-diagnostics} show representative \(p=7\), \(n_{\rm edge}=110\) computations at \(t=3.7\) and \(t=6.7\). Each computation uses \(1{,}379{,}400\) degrees of freedom, which is close to the resolution used in \cite{bolm2026invariant}. In both formulations, regions of appreciable \(\theta_k\) generally coincide with non-negligible \(\varepsilon_k\) and follow the thin shear layers and developing vortical structures. The correspondence is stronger for D-SBP. For C-SBP, \(\theta_k\) has a somewhat broader low-level response than \(\varepsilon_k\), indicating that the nonlinear activation and entropy-viscosity factor suppress much of the background sensor response. The stronger correspondence for D-SBP is consistent with the LLF interface dissipation damping some of the oscillatory content. The density fields obtained with both formulations qualitatively resemble the recently published result in \cite{bolm2026invariant}, which uses adaptive mesh refinement and a positivity-preserving entropy-stable discretization on quadrilateral meshes. Together with the failure study, this qualitative agreement suggests that the robustness improvement is obtained without adding excessive dissipation, retaining the characteristic roll-up and vortical structures. Additional comparisons up to \(t_f=1.5\) using matched adaptive time-integration settings showed no substantial increase in accepted-step counts when the proposed dissipation was added.

\begin{figure}[!t]
	\centering
	\begin{subfigure}{0.95\textwidth}
		\includegraphics[width=\linewidth]{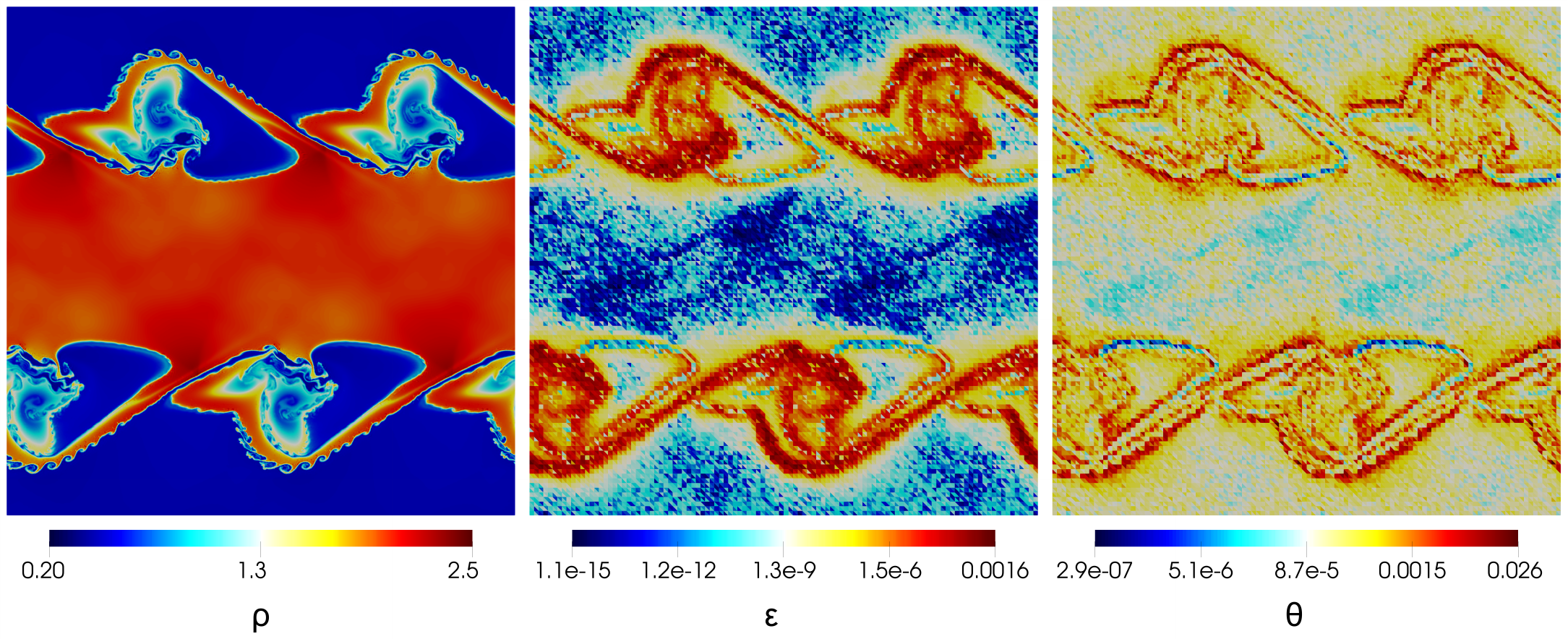}
		\caption{\(t=3.7\).}
	\end{subfigure}
	\begin{subfigure}{0.95\textwidth}
		\includegraphics[width=\linewidth]{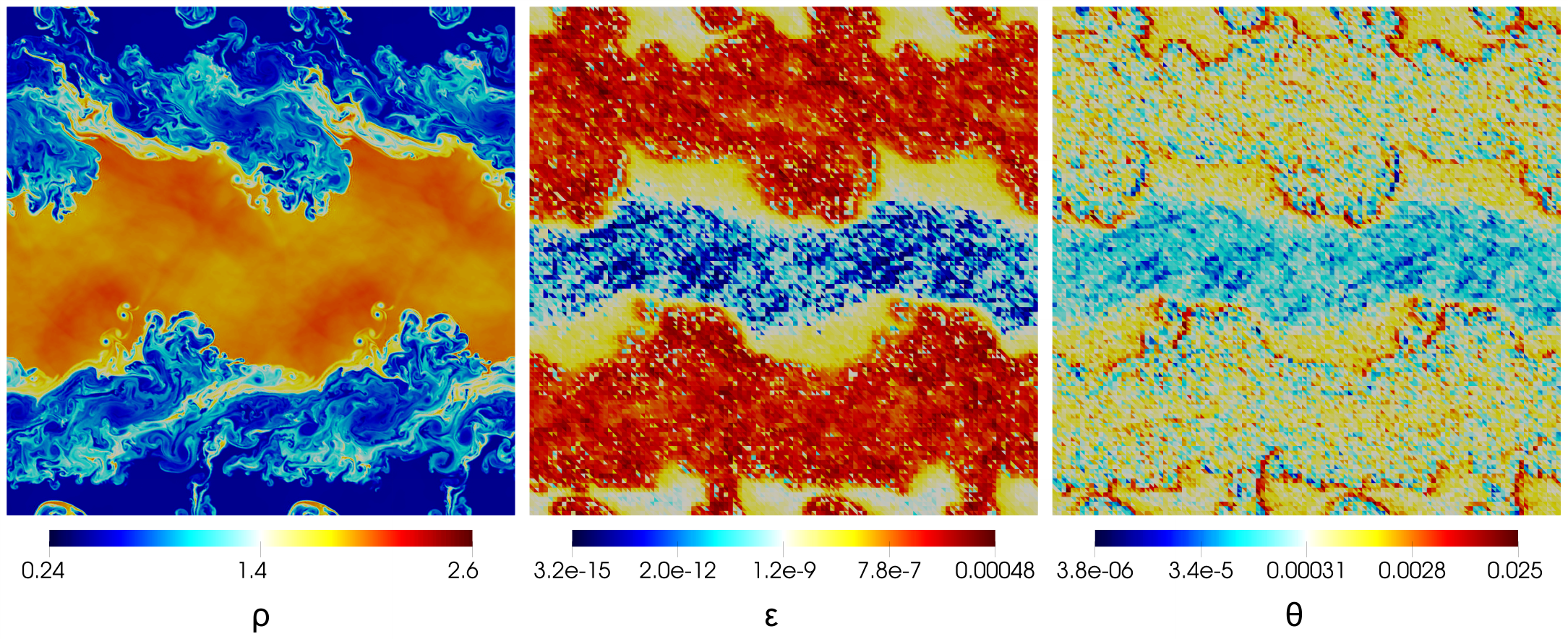}
		\caption{\(t=6.7\).}
	\end{subfigure}
	\caption{Kelvin--Helmholtz density, entropy budget \(\varepsilon_k\), and smoothness sensor \(\theta_k\) for C-SBP discretization with the proposed artificial dissipation, using \(p=7\) and \(n_{\rm edge}=110\). Note that the colormaps for \(\varepsilon_k\) and \(\theta_k\) use logarithmic scaling with values below \(10^{-16}\) clipped.} \label{fig:kh-csbp-diagnostics}
\end{figure}

\begin{figure}[!t]
	\centering
	\begin{subfigure}{0.95\textwidth}
		\includegraphics[width=\linewidth]{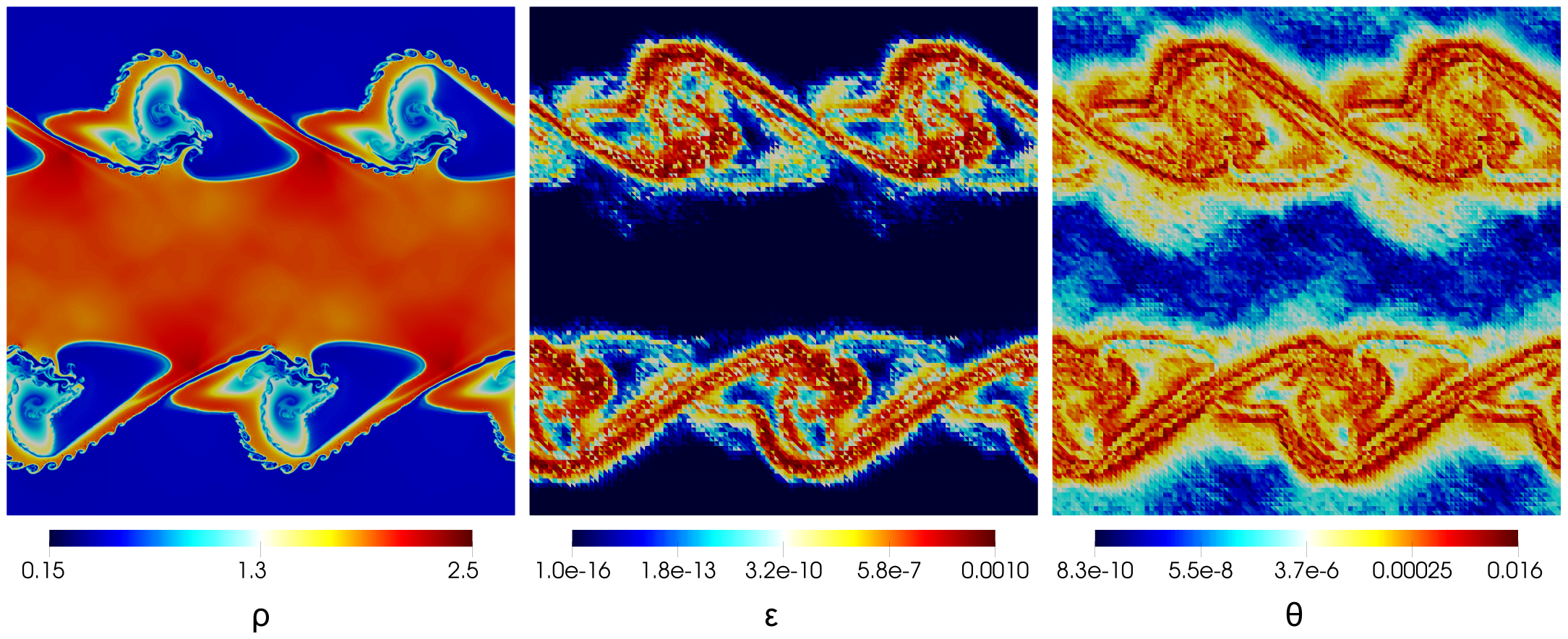}
		\caption{\(t=3.7\).}
	\end{subfigure}
	\begin{subfigure}{0.95\textwidth}
		\includegraphics[width=\linewidth]{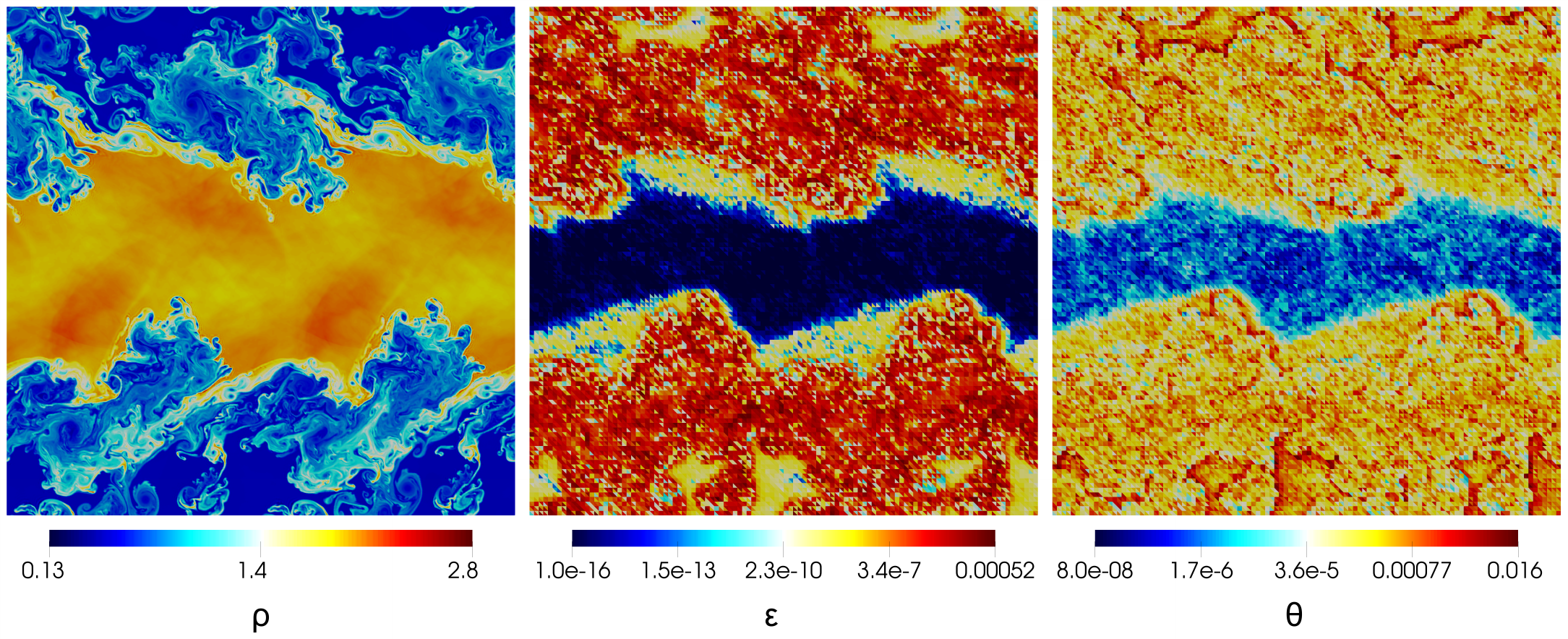}
		\caption{\(t=6.7\).}
	\end{subfigure}
	\caption{Kelvin--Helmholtz density, entropy budget \(\varepsilon_k\), and smoothness sensor \(\theta_k\) for D-SBP with LLF discretization and the proposed artificial dissipation, using \(p=7\) and \(n_{\rm edge}=110\). Note that the colormaps for \(\varepsilon_k\) and \(\theta_k\) use logarithmic scaling with values below \(10^{-16}\) clipped.} \label{fig:kh-dsbp-diagnostics}
\end{figure}

\subsection{Two-dimensional Riemann problem}
The two-dimensional Riemann problem \cite{lax1998solution,kurganov2002solution,liska2003comparison} is a multidimensional shock-interaction test case. Four constant states meet at the origin, and the resulting flow contains interacting shocks and contact discontinuities. Following \cite{chan2018discretely}, we solve the Euler equations on the enlarged periodic computational domain \(\Omega=[-1,1]^2\) until \(t_f=0.25\). The primitive variables are initialized by four constant quadrant states,
\begin{equation}\label{eq:riemann2d_ic}
	(\rho,u,v,p)(x,y,0)
	=
	\begin{cases}
		(0.5313,\;0,\;0,\;0.4), & x\ge0,\; y\ge0, \\
		(1,\;0.7276,\;0,\;1),   & x<0,\; y\ge0,   \\
		(0.8,\;0,\;0,\;1),      & x<0,\; y<0,     \\
		(1,\;0,\;0.7276,\;1),   & x\ge0,\; y<0 .
	\end{cases}
\end{equation}

This problem is a severe positivity-sensitive stress test. The proposed artificial dissipation is entropy dissipative, but it is not a positivity-preserving limiter. For this test, we apply the artificial dissipation to the direct divergence-form base spatial discretization, which is not provably entropy stable, since the entropy-stable discretizations do not complete the simulation due to violation of positivity requirements. For D-SBP operators, the LLF interface dissipation is used along with the proposed artificial dissipation. The results below therefore test the capability and implementation flexibility of the proposed artificial dissipation for problems with shock waves.

\begin{figure}[!t]
	\centering
	\begin{subfigure}{0.95\textwidth}
		\includegraphics[width=\linewidth]{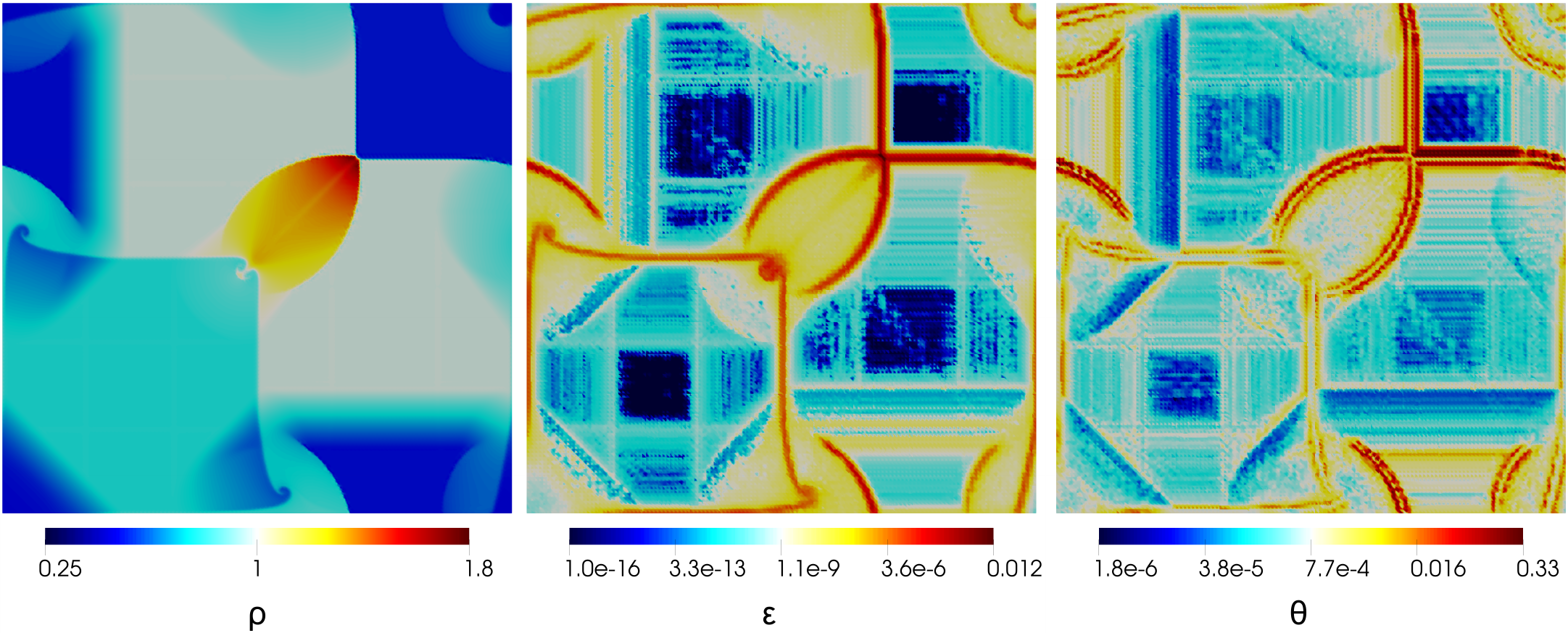}
		\caption{C-SBP discretization with AD.}
	\end{subfigure}
	\begin{subfigure}{0.95\textwidth}
		\includegraphics[width=\linewidth]{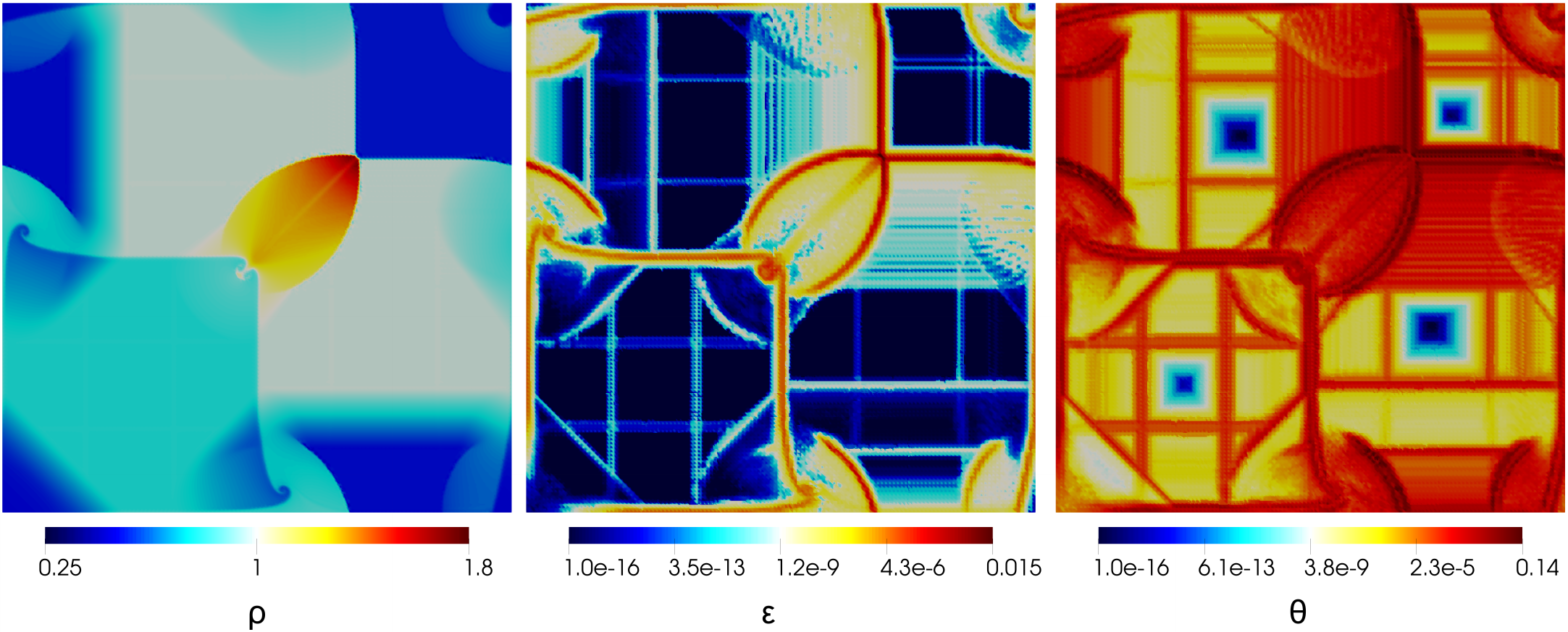}
		\caption{D-SBP discretization with LLF and AD.}
	\end{subfigure}
	\caption{Two-dimensional Riemann problem using \(p=3\), \(n_{\rm edge}=128\), the direct divergence-form base spatial discretization, and the proposed artificial dissipation. Each row shows the density, the entropy budget \(\varepsilon_k\), and the sensor \(\theta_k\) at \(t_f=0.25\). Note that the colormaps for \(\varepsilon_k\) and \(\theta_k\) use logarithmic scaling with values below \(10^{-16}\) clipped.}
	\label{fig:riemann2d-div-diagnostics}
\end{figure}

\Cref{fig:riemann2d-div-diagnostics} shows that both computations reach the final time and produce closely agreeing density fields. In both formulations, \(\theta_k\) responds strongly near the shock-interaction structures but also to small oscillations away from the discontinuities. The entropy budget \(\varepsilon_k\) is more sharply concentrated near the shocks, indicating that the entropy-viscosity factor, \(b_{\nu,k}\), suppresses much of the broader sensor response. Thus, the proposed dissipation retains its targeted response for both C-SBP and D-SBP divergence-form discretizations.

\subsection{Three-dimensional inviscid Taylor--Green vortex}
The inviscid Taylor--Green vortex is included as a three-dimensional under-resolved vortical-flow test. The preceding tests demonstrate improved robustness and performance for solutions containing shocks. Since excessive dissipation can suppress physically relevant flow structures, we next examine the damping introduced by the proposed dissipation in the inviscid Taylor--Green vortex through its effect on kinetic-energy evolution. We solve the three-dimensional Euler equations on \(\Omega=[-\pi,\pi]^3\) with periodic boundary conditions until \(t_f=20\). We set \(u_0=\rho_0=R=L=1\), \(\mathrm{Ma}=0.1\), and 
\[
	p_0=\frac{\rho_0u_0^2}{\gamma\mathrm{Ma}^2},
	\qquad T_0=\frac{p_0}{R\rho_0}.
\]
Since the flow is nearly incompressible at \(\mathrm{Ma}=0.1\) and the Ranocha volume flux is kinetic-energy preserving, the kinetic energy is expected to remain nearly constant in the absence of added spatial dissipation. Using identical time-integration settings, we compare the effects of the different stabilization choices on kinetic-energy evolution. The initial condition in terms of the primitive variables is given by 
\begin{align}
	u& = u_0 \sin\!\left(\frac{x}{L}\right)	\cos\!\left(\frac{y}{L}\right)\cos\!\left(\frac{z}{L}\right),
	\qquad
	v =	-u_0 \cos\!\left(\frac{x}{L}\right)	\sin\!\left(\frac{y}{L}\right)\cos\!\left(\frac{z}{L}\right),
	\qquad
	w=0,
	\nonumber\\
	p&=	p_0	+ \frac{\rho_0u_0^2}{16} \left[\cos\!\left(\frac{2x}{L}\right)
	+ \cos\!\left(\frac{2y}{L}\right)\right] \left[	\cos\!\left(\frac{2z}{L}\right)+2\right],
	\label{eq:euler3d_tgv_ic}
	\qquad
	\rho =\frac{p}{R T_0}.
	\nonumber
\end{align}
We compute the kinetic energy,  
\[
	E_K(t)	= \frac{1}{2|\Omega|} \int_{\Omega}\rho(u^2+v^2+w^2)\,\dd\Omega,
\]
and its dissipation rate, \(-{\dd E_K}/{\dd t}\). The dissipation rate is approximated using finite differences of the kinetic-energy time series.

\begin{figure}[!t]
	\centering
	\begin{subfigure}{0.48\textwidth}
		\includegraphics[width=\linewidth]{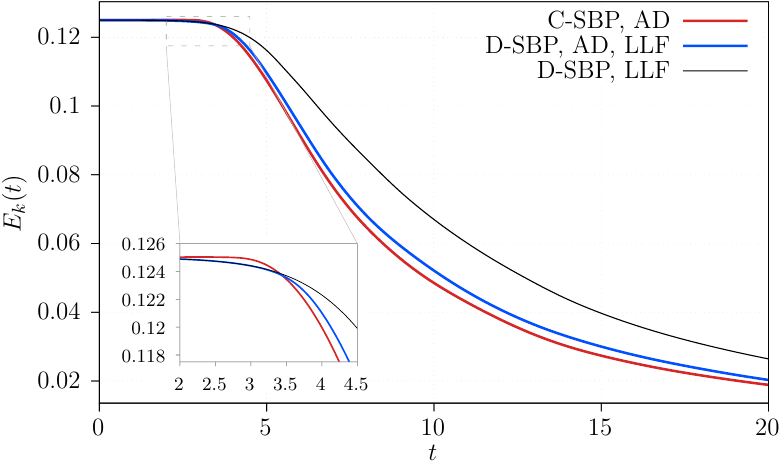}
		\caption{\(p=3\), kinetic energy.}
	\end{subfigure}
	\begin{subfigure}{0.48\textwidth}
		\includegraphics[width=\linewidth]{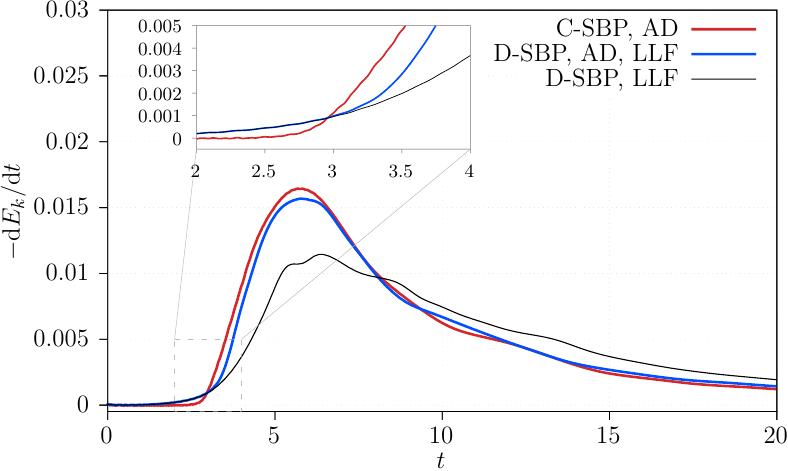}
		\caption{\(p=3\), kinetic-energy dissipation rate.}
	\end{subfigure}
	\begin{subfigure}{0.48\textwidth}
		\includegraphics[width=\linewidth]{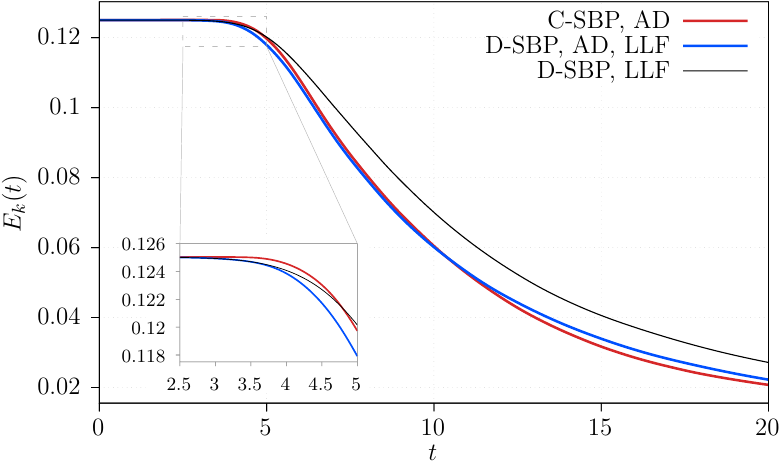}
		\caption{\(p=4\), kinetic energy.}
	\end{subfigure}
	\begin{subfigure}{0.48\textwidth}
		\includegraphics[width=\linewidth]{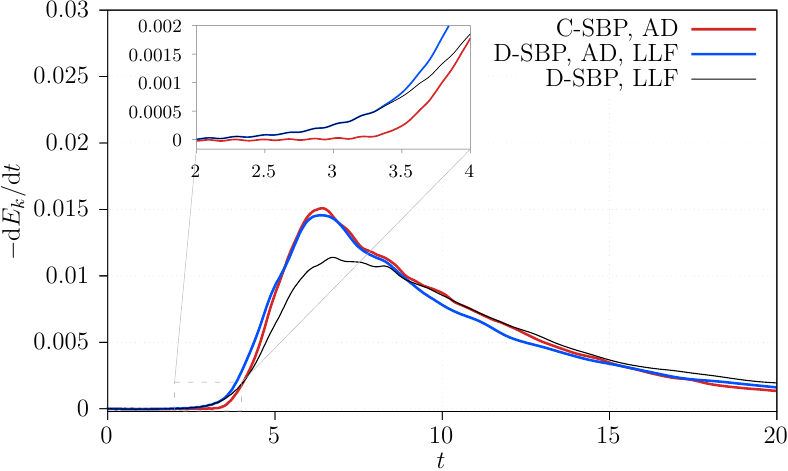}
		\caption{\(p=4\), kinetic-energy dissipation rate.}
	\end{subfigure}
	\caption{Inviscid Taylor--Green vortex kinetic energy and kinetic-energy dissipation rate for \(p=3,4\) and \(n_{\rm edge}=6\).}
	\label{fig:tgv3d-beta-comparison}
\end{figure}

\Cref{fig:tgv3d-beta-comparison} shows that, for both \(p=3\) and \(p=4\), C-SBP with AD has a lower initial kinetic-energy dissipation rate than D-SBP with LLF for \(t\lesssim3\). The D-SBP histories with and without AD initially nearly coincide, indicating little additional damping from AD during this phase. At later times, the two AD histories remain close and exhibit larger peak dissipation rates and greater total kinetic-energy decay than D-SBP with LLF alone. These later differences also reflect the response of the AD to larger oscillations as the solution evolves.

\section{Conclusion}\label{sec:conclusion}
In this paper, we have developed an element-local volume dissipation with separately controlled spatial distribution, magnitude, activation, and convergence rate under mesh refinement. The formulation is discretization agnostic, applying to continuous and discontinuous SBP and compatible Galerkin methods without depending on the form of the base spatial residual. The added term has zero conservative moment and nonpositive entropy contribution independent of the base method, thereby preserving conservation and semidiscrete entropy stability whenever the base method satisfies these properties. Moreover, changing the admissible activation parameters, local dissipation orders, or nonnegative entropy budget does not invalidate these guarantees. For smooth admissible solutions on shape-regular, quasi-uniform mesh families with affine element mappings, the assembled dissipation is \(\fn{O}(h^{p_{\rm L}})\), improving to \(\fn{O}(h^{p_{\rm H}})\) for \(p\geq2\) and \(p_{\rm L}\le p_{\rm H}\). With suitable order choices, formal accuracy can therefore be retained for smooth solutions. For implementation, \cref{app:implementation_summary} collects the core formulas and the parameter choices used in the numerical experiments.

The two- and three-dimensional isentropic-vortex tests show that the added dissipation retains high-order solution convergence and leaves the smooth long-time solution errors essentially unchanged. The proposed dissipation substantially improves robustness for the Kelvin--Helmholtz instability: all C-SBP and D-SBP computations with \(p\leq7\) reach the final time, whereas every D-SBP computation with LLF interface dissipation alone fails. The representative density fields also retain the characteristic roll-up and vortical structures. In the computations involving shocks, the entropy budgets are concentrated near the discontinuities. The Taylor--Green tests show lower initial kinetic-energy dissipation for C-SBP with AD than for D-SBP with LLF for \(t\lesssim3\), and little initial influence of AD when applied to D-SBP with LLF. These results are obtained using the same dissipation parameter-selection rules across the test problems. Taken together, they demonstrate that the proposed targeted artificial dissipation can substantially improve robustness without excessive damping of resolved flow structures.

Failures remain for runs with \(p\geq8\) in the Kelvin--Helmholtz tests, as the proposed dissipation does not guarantee positivity. The tested entropy-stable discretizations also fail in the two-dimensional Riemann problem even with AD because of positivity violations. Adapting the prescribed dissipation parameters and combining the operator with positivity-preserving limiting are therefore important directions for future work.

\appendix
\crefalias{section}{appendix}
\section{Relation to Existing Volume Dissipation Schemes} \label{app:alternative_choices}

Recall the general structure of our volume dissipation given by \eqref{eq:local_artificial_dissipation_update} and \eqref{eq:local_alpha_definition},
\begin{equation*} 
\der[\bm u_g]{t}=\bm r_g+\overline{\H}_g^{-1} \sum_{\Omega_k\in\mathcal{T}_h} \overline{\P}_k^T\bm d_k, 
\qquad \bm d_k \coloneqq - \bm\alpha_{k,\epsilon}\varepsilon_k, 
\end{equation*}
where $\varepsilon_k = \kappa_{\rm diss}\theta_k^{\beta_k} b_{\nu,k}$ specifies the total amount of volume dissipation in each element, and
\begin{equation*}
\bm\alpha_{k,\epsilon} \coloneqq
\frac{b_k\bm a_k}{b_k^2+\epsilon_{\rm mach}V_k^2},
\qquad b_k \coloneqq\bm w_k^T\bm a_k,
\qquad \epsilon_{\rm mach}>0 .
\end{equation*}

\subsection{Dissipation of Bercik, Craig Penner, and Zingg}

Consider now the artificial volume dissipation operators of Bercik, Craig Penner, and Zingg~\cite{bercik2026stable} acting on entropy variables,
\begin{equation} \label{eq:bercik_diss_main_form} 
\der[\bm u_g]{t}=\bm r_g - \kappa_{\rm diss} \overline{\H}_g^{-1} \sum_{\Omega_k\in\mathcal{T}_h} \overline{\P}_k^T \sum_{i=1}^d \left( \overline{\mathsf{D}}^{s_a}_{\xi_i} \right)^T \left( \mathsf{C}_i \otimes \mathsf{A}_{0,k} \right) \overline{\mathsf{D}}^{s_a}_{\xi_i} \bm w_k , 
\end{equation}
where $\mathsf{C}_i$ contains some boundary correction and approximation to the maximum wave speeds under the condition that $\left( \mathsf{C}_i \otimes \mathsf{A}_{0,k} \right) \succeq 0$ is symmetric positive semidefinite. As was suggested in \cref{rem:ref2phys}, here the reference-to-physical scaling is handled implicitly by metric terms in $\mathsf{C}_i$. Note also that in~\cite{bercik2026stable}, $\overline{\mathsf{D}}^{s_a}_{\xi_i}$ is defined as an undivided finite difference derivative operator, but for element-type SBP operators, this is equivalent (up to a constant) to the convention in this work of the difference operator defined on the reference element. To place this operator in the current framework, set \(\theta_k=1\) and choose
\begin{equation} \label{eq:bercik_diss_new_form}
\bm a_k \coloneqq \sum_{i=1}^d \left( \overline{\mathsf{D}}^{s_a}_{\xi_i} \right)^T \left( \mathsf{C}_i \otimes \mathsf{A}_{0,k} \right) \overline{\mathsf{D}}^{s_a}_{\xi_i} \bm w_k , \qquad
b_{\nu,k} \coloneq \bm w_k^T \bm a_k .
\end{equation}
The complete update then satisfies
\(\bm d_k=-\kappa_{\rm diss}b_k^2\bm a_k/(b_k^2+\epsilon_{\rm mach}V_k^2)\). If \(b_k>0\), then \(\bm d_k\to-\kappa_{\rm diss}\bm a_k\) in the exact-arithmetic limit \(\epsilon_{\rm mach}\to0^+\). If \(b_k=0\), positive semidefiniteness implies \(\bm a_k=\bm0\), and therefore \(\bm d_k=\bm0=-\kappa_{\rm diss}\bm a_k\) for every \(\epsilon_{\rm mach}>0\). Thus, \cref{eq:bercik_diss_main_form} is recovered as \(\epsilon_{\rm mach}\to0^+\).
This is similar in structure to the choices made in \eqref{eq:scaled_reference_distribution_a} for $\bm a_k$, and \eqref{eq:epsilon_k} and \eqref{eq:bnu_cheap} for $\varepsilon_k$, but with a few notable differences. First, for minimally nodal spectral elements, the highest nontrivial distribution order \(s_a=p\) restricts the dissipation to the highest-degree polynomial mode and relaxes polynomial exactness by one degree~\cite{bercik2026stable}. The method proposed here instead fixes \(s_a=1\) and uses the separate sensor order \(s_\theta\) together with \(\theta_k^{\beta_k}\) to obtain the prescribed smooth-state decay. Second, in \eqref{eq:bercik_diss_new_form}, the same entropy variables $\bm w_k$ are used for both $\bm a_k$ and $b_{\nu,k}$, allowing one to collapse to the simpler linear form \eqref{eq:bercik_diss_main_form}. Instead, as explained in section~\ref{sec:entropy_stable_artificial_dissipation_form}, in this work we use $\bm u_k$ in $\bm a_k$ for accuracy reasons, and separately use $\bm w_k$ in the budget term $b_{\nu,k}$ to mimic an entropy viscosity.

\subsection{Local Projection Stabilization}

Local projection stabilization (LPS) is another technique that involves penalizing high-order polynomial degrees in the solution. Here, we explicitly consider the adaptation of LPS to the entropy-stable SBP framework by Hicken \cite{hicken2020entropy}, which takes the form
\begin{equation} \label{eq:lps} 
\der[\bm u_g]{t}=\bm r_g - \kappa_{\rm diss} \overline{\H}_g^{-1} \sum_{\Omega_k\in\mathcal{T}_h} \overline{\P}_k^T \overline{\mathsf{\Pi}}^T \overline{\H}_\xi \mathsf{C}_k \overline{\mathsf{\Pi}} \bm w_k , 
\end{equation}
where $\mathsf{\Pi}$ isolates the high-frequency content in $\bm w_k$ through a discrete $L^2$ projection (or more specifically, isolates the component outside the degree-$p$ polynomial space), and $\mathsf{C}_k \succeq 0$ again contains some symmetric positive semidefinite approximation to the maximum wave speeds and change-of-variable Jacobian matrix $\A_0$.
The same complete-update limiting argument recovers \cref{eq:lps} by setting \(\theta_k=1\) and choosing
\begin{equation} \label{eq:lps_new_form}
\bm a_k \coloneqq \overline{\mathsf{\Pi}}^T \overline{\H}_\xi \mathsf{C}_k \overline{\mathsf{\Pi}} \bm w_k , \qquad
b_{\nu,k} \coloneq \bm w_k^T \bm a_k .
\end{equation}
However, this general framework has the advantage that, unlike in \cite{hicken2020entropy}, one need not apply the projection directly to the entropy variables to ensure entropy stability. We can for example instead take
\[
\bm a_k \coloneqq \overline{\mathsf{\Pi}}^T \overline{\H}_\xi \overline{\mathsf{\Pi}} \bm u_k
\]
together with any appropriate nonnegative scalar budget $b_{\nu,k}$ and bounded smoothness sensor $\theta_k$, while still being provably entropy-dissipative. Changing the variables alone does not remove the nodal restriction: a projection that annihilates the full degree-\(p\) space is zero on a minimally nodal element. Avoiding this restriction requires a distribution that acts on lower-degree variations, together with a nonlinear budget that supplies the required smooth-state decay, as in the first-derivative construction used here.

\subsection{Residual-based artificial viscosity}

Residual-based viscosity methods can also be adapted to fit within the current framework. Although common in the finite-element community, here we focus our comparison to the implementations within the SBP framework performed by Stiernstr{\"o}m \etal~\cite{STIERNSTROM} and Niemel{\"a}~\cite{niemela2026residual}. These methods introduce a variable viscosity coefficient of the form 
\begin{equation}
    \nu_k = \min\left(\nu_{\max,k},\nu_{{\rm RV},k}\right),
\end{equation}
where the first-order and high-order viscosities are separated as
\[
\nu_{\max,k}=C_1\lambda_k h_k,
    \qquad
    \nu_{{\rm RV},k}
    = C_{\rm RV}h_k^2 \frac{\abs{\mathcal{R}(\bm u_k)}}{N(\bm u_k)} ,
\]
and $C_1$ and $C_{\rm RV}$ are constants, $\lambda_k$ is some approximation to the maximum wavespeed, $h_k$ is the mesh spacing, $\mathcal{R}$ is some approximation to the residual, and $N$ is a suitable normalization. We defer to \cite{STIERNSTROM, niemela2026residual} for details on how $\mathcal{R}$ and $N$ can be computed. This viscosity coefficient then multiplies some appropriate viscous flux. 

Within the present framework, it is natural to interpret the residual viscosity primarily as a smoothness sensor. By setting \(\beta_k=1\), using \(\theta_k\) directly as the activation, and defining
\begin{equation}
    \theta_k
    \coloneqq
    \min\left(1,\frac{\nu_{{\rm RV},k}}{\nu_{\max,k}}\right),
\end{equation}
the first-order viscosity can then be used to compute the entropy budget, similar to our current choice in \eqref{eq:bnu_cheap}. For example, we may define a budget
\begin{equation}
    b_{\nu,k}
    \coloneqq
    \nu_{\max,k}\mu_k
    \sum_{i=1}^{d}
    \left(\overline{\D}_{+,\xi_i}\bm w_k\right)^T
    \left(\mathsf{B}_i \H_{\xi}\otimes\bar{\A}_{0,k}\right)
    \left(\overline{\D}_{+,\xi_i}\bm w_k\right),
\end{equation}
and a distribution vector described by
\begin{equation} \label{eq:app_RV_a}
    \bm a_k
    =
    \sum_{i=1}^{d}
    \overline{\D}_{+,\xi_i}^{T} \overline{\mathsf{B}}_i
    \overline{\H}_{\xi}
    \overline{\D}_{+,\xi_i} \bm u_k,
\end{equation}
where $\D_{+,\xi_i}$ is an upwind SBP operator (e.g.\,see~\cite{niemela2026residual}), and $\mathsf{B}_i$ is some boundary modification that facilitates the imposition of boundary conditions, e.g.\, in one dimension given by $\operatorname{diag} (0,1,\dots,1,0)$. If $\mathsf{B}_i$ is instead placed on the left-hand side of the operator in \eqref{eq:app_RV_a}, this boundary modification would allow us to rewrite a direct application of a second derivative as in \cite{niemela2026residual} in quadratic form since interface terms vanish after integration by parts, thereby more closely resembling the distribution vector \eqref{eq:scaled_reference_distribution_a} used in this work. In contrast to \cite{niemela2026residual}, however, here we have moved $\mathsf{B}_i$ inside the operator to ensure that the conservation property $\overline{\I}_k^T\bm\alpha_{k,\epsilon}=\bm0$ holds.

For a scalar problem with $\bm w=\bm u$ and $\A_0=1$, in the unregularized limit, the above construction reduces to the familiar scalar residual-viscosity diffusion. For nonlinear systems, the entropy normalization and the positive entropy budget introduced by this framework provide an entropy-stable modification of the monolithic conservative-variable residual viscosity operator, that is, a single viscosity coefficient applied across the various conservative variables. The construction of \cite{niemela2026residual}, by contrast, employs a physics-based viscosity for the Navier--Stokes equations that treats the components separately, and would therefore correspond to further generalizations of this framework.

The original residual viscosity methods also allow the viscosity $\nu_k$ to vary from node to node, while above we have simplified it to be a scalar in each element. Such variation, however, can be retained by modifying the distribution vector. Let $\vartheta_{k,m}$ denote a nodal residual activation and set
\[
    \theta_k=\max_m\vartheta_{k,m},
    \qquad
    \widehat{\vartheta}_{k,m}
    =\begin{cases}
        \frac{\vartheta_{k,m}}{\theta_k}, & \text{if $\theta_k > \epsilon_{\rm mach}$}\\
        0, & \text{if $\theta_k \leq \epsilon_{\rm mach}$}
    \end{cases},
    \qquad
    0\leq\widehat{\vartheta}_{k,m}\leq1.
\]
The scalar $\theta_k$ then controls the overall amount of residual activation, while the normalized nodal variation can be incorporated into the distribution and budget through
\[
    \mathsf{A}_k^{\rm RV}
    \coloneqq
    \operatorname{diag}\left(
    \widehat{\vartheta}_{k,1},\ldots,
    \widehat{\vartheta}_{k,n_p}\right).
\]
For example, $\overline{\H}_{\xi}$ can be replaced by
$\left( \H_{\xi}\mathsf{A}_k^{\rm RV} \right) \otimes \I_{n_c}$ in both of the preceding quadratic forms. Since $\mathsf{A}_k^{\rm RV}\succeq0$, this retains the conservative and entropy-dissipative structure while reproducing the spatial localization characteristic of nodal residual viscosity methods.

\section{Implementation Summary}\label{app:implementation_summary}

\setlength{\tabcolsep}{5pt}
\renewcommand{\arraystretch}{1.12}
\begin{center}
\begin{tabular}{@{}>{\arraybackslash}p{0.065\textwidth}p{0.45\textwidth}p{0.4\textwidth}@{}}
\toprule
\textbf{Eq.} & \textbf{Core equations} & \textbf{Related equations} \\
\midrule
\eqref{eq:local_artificial_dissipation_update}
& \(\displaystyle \bm d_k=-\bm\alpha_{k,\epsilon}\varepsilon_k\)
&
\(\bm\alpha_{k,\epsilon}\) \cref{eq:local_alpha_definition};
\(\varepsilon_k\) \cref{eq:epsilon_k} \\
\addlinespace[0.75em]

\eqref{eq:local_alpha_definition}
&
\(\displaystyle
\bm\alpha_{k,\epsilon}
=\frac{b_k\bm a_k}{b_k^2+\epsilon_{\rm mach}V_k^2}
\)
&
\(b_k\) \cref{eq:b_k};
\(\bm a_k\) \cref{eq:scaled_reference_distribution_a};
\(V_k\) \cref{eq:element_volume} \\
\addlinespace[0.75em]

\eqref{eq:b_k}
& \(\displaystyle b_k=\bm w_k^T\bm a_k\)
&
\(\bm w_k\) \cref{eq:entropy_variables};
\(\bm a_k\) \cref{eq:scaled_reference_distribution_a} \\
\addlinespace[0.75em]

\eqref{eq:scaled_reference_distribution_a}
&
\(\displaystyle
\bm a_k=\M_k\bm u_k,
\quad
\M_k=\mu_k\hat{\M}_{s_a}
\)
&
\(\hat{\M}_s\) \cref{eq:reference_smoothness_matrix};
\(\mu_k\) \cref{eq:distribution_scaling_factor} \\
\addlinespace[0.75em]

\eqref{eq:epsilon_k}
&
\(\displaystyle
\varepsilon_k=\kappa_{\rm diss}\theta_k^{\beta_k}b_{\nu,k}
\)
&
\(\theta_k\) \cref{eq:volume_sensor_operator_normalized_fluctuation};
\(\beta_k\) \cref{eq:sensor_exponent};
\(b_{\nu,k}\) \cref{eq:bnu_cheap} \\
\addlinespace[0.75em]

\eqref{eq:volume_sensor_operator_normalized_fluctuation}
&
\(\displaystyle
\theta_k=
\frac{S_k}{\Lambda_{s_\theta}
\left[N'_k+\widehat V(h_{{\rm eff},k}/L_{\rm ref})^2\right]}
\)
&
\makecell[l]{
\(S_k\) \cref{eq:volume_sensor_numerator};
\(N'_k\) \cref{eq:volume_sensor_fluctuation_norm};
\(\widetilde{\bm u}'_k\) \cref{eq:scaled_state_fluctuation};
\\
\(\hat{\M}_{s_\theta}\) \cref{eq:reference_smoothness_matrix};
\(\Lambda_{s_\theta}\) \cref{eq:Lambda_s};
\(h_{{\rm eff},k}\) \cref{eq:element_volume}
 } \\
\addlinespace[0.75em]

\eqref{eq:theta_exponent_indicator}
&
\(\displaystyle
\chi_k=
\frac{\theta_k^{m_\theta}}
{\theta_k^{m_\theta}+c_\theta^{m_\theta}}
\)
&
\(\theta_k\) \cref{eq:volume_sensor_operator_normalized_fluctuation} \\
\addlinespace[0.75em]

\eqref{eq:local_dissipation_order}
&
\(\displaystyle
p_k=(1-\chi_k)p_{\rm H}+\chi_kp_{\rm L}
\)
&
\(\chi_k\) \cref{eq:theta_exponent_indicator} \\
\addlinespace[0.75em]

\eqref{eq:sensor_exponent}
&
\(\displaystyle
\beta_k=\frac{p_k}{2s_\theta-2}
\)
&
\(p_k\) \cref{eq:local_dissipation_order} \\
\addlinespace[0.75em]

\bottomrule
\end{tabular}
\end{center}

\vspace{0.5cm}

\newpage
\noindent\textbf{Artificial-dissipation parameter choices}
\begin{center}
\setlength{\tabcolsep}{4pt}
\begin{tabular}{@{}>{\arraybackslash}p{0.075\textwidth}>{\raggedright\arraybackslash}p{0.525\textwidth}>{\centering\arraybackslash}p{0.075\textwidth}>{\raggedright\arraybackslash}p{0.25\textwidth}@{}}
\toprule
\textbf{Eq.} & \textbf{Parameter} & \textbf{Eq.} & \textbf{Parameter} \\
\midrule
\eqref{eq:epsilon_k}
& \(\kappa_{\rm diss}=1\)
& \eqref{eq:scaled_reference_distribution_a}
& \(s_a=1\) \\
\eqref{eq:s_theta}
& \(s_\theta=\lceil p/2\rceil+1\;(p\geq2)\)
& \eqref{eq:volume_sensor_operator_normalized_fluctuation}
& \(L_{\rm ref}=1\)\\
\eqref{eq:local_dissipation_order}
& \(p_{\rm L}=p\) 
&\eqref{eq:local_dissipation_order}
& \(p_{\rm H}=2p\) \\
\eqref{eq:theta_exponent_indicator}
& \(c_\theta=10^{-4}\) 
&\eqref{eq:theta_exponent_indicator}
& \(m_\theta=4\)
\\
\bottomrule
\end{tabular}
\end{center}

\section*{CRediT authorship contribution statement}
\textbf{Zelalem Arega Worku:} Conceptualization, Methodology, Formal analysis, Software, Investigation, Validation, Visualization, Writing -- original draft, Writing -- review \& editing. 
\textbf{Alex Bercik:} Methodology, Formal analysis, Investigation, Software, Validation, Writing -- original draft (appendix), Writing -- review \& editing.
\textbf{David C. Del Rey Fern{\'a}ndez:} Supervision,  Validation, Writing -- review \& editing.
\textbf{David W. Zingg:} Supervision, Validation, Writing -- review \& editing.

\section*{Data availability}
Data are available from the corresponding author upon reasonable request.

\section*{Declaration of generative AI and AI-assisted
technologies in the manuscript preparation process}
OpenAI Codex was used for coding assistance and to identify and address writing and clarity issues in the manuscript. The authors wrote and reviewed the entire manuscript and take full responsibility for its content.

\bibliographystyle{ieeetr}
\bibliography{ref}

\end{document}